\documentclass[reqno]{amsart}
\usepackage{amsfonts,amscd,amsthm,amsgen,amsmath,amssymb,esint}
\usepackage[all]{xy}
\usepackage{epsfig,color}
\usepackage[vcentermath]{youngtab}
\usepackage{enumitem}
\newlist{abbrv}{itemize}{1}
\setlist[abbrv,1]{label=,labelwidth=1.7in,align=parleft,itemsep=0.1\baselineskip,leftmargin=!}
\newtheorem{theorem}{Theorem}[section]
\newtheorem{lemma}[theorem]{Lemma}

\newtheorem{proposition}[theorem]{Proposition}

\theoremstyle{definition}
\newtheorem{definition}[theorem]{Definition}

\theoremstyle{remark}
\newtheorem{remark}[theorem]{Remark}

\numberwithin{equation}{section}

\begin{document}

\title{Tautological classes of definite $4$-manifolds}
\author{David Baraglia}

\address{School of Mathematical Sciences, The University of Adelaide, Adelaide SA 5005, Australia}

\email{david.baraglia@adelaide.edu.au}


\date{\today}

\begin{abstract}
We prove a diagonalisation theorem for the tautological, or generalised Miller--Morita--Mumford classes of compact, smooth, simply-connected definite $4$-manifolds. Our result can be thought of as a families version of Donaldson's diagonalisation theorem. We prove our result using a families version of the Bauer--Furuta cohomotopy refinement of Seiberg--Witten theory. We use our main result to deduce various results concerning the tautological classes of such $4$-manifolds. In particular, we completely determine the tautological rings of $\mathbb{CP}^2$ and $\mathbb{CP}^2 \# \mathbb{CP}^2$. We also derive a series of linear relations in the tautological ring which are universal in the sense that they hold for all compact, smooth, simply-connected definite $4$-manifolds.
\end{abstract}

\maketitle


\section{Introduction}\label{sec:intro}

\subsection{Tautological classes}

Let $X$ be a compact, simply-connected smooth $4$-manifold with positive definite intersection form. Assume that $b_2(X) > 0$. Then by the work of Donaldson \cite{don} and Freedman \cite{fre}, $X$ is homeomorphic to the connected sum $\#^n \mathbb{CP}^2$ of $n \ge 1$ copies of $\mathbb{CP}^2$, where $n = b_2(X)$.

Let $\pi : E \to B$ be a compact, smooth family with fibres diffeomorphic to $X$. By this we mean that $E$ and $B$ are compact smooth manifolds, $\pi$ is a proper submersion and each fibre of $\pi$ with its induced smooth structure is diffeomorphic to $X$. Note that $E$ has a fibrewise orientation which is uniquely determined by the requirement that the fibres of $E$ are positive definite $4$-manifolds. In this paper, we will use parametrised Seiberg--Witten theory to study the {\em tautological classes}, or {\em generalised Miller--Morita--Mumford classes} of such families. These are defined as follows. Let $T(E/B) = Ker(\pi_* : TE \to TB)$ denote the vertical tangent bundle. Then for each rational characteristic class $c \in H^*( BSO(4) ; \mathbb{Q})$, we define the associated tautological class as
\[
\kappa_c(E) = \int_{E/B} c( T(E/B) ) \in H^{*-4}(B ; \mathbb{Q})
\]
where $\int_{E/B}$ denotes integration over the fibres. Let $Diff(X)$ denote the group of diffeomorphisms of $X$ with the $\mathcal{C}^\infty$-topology (note that all diffeomorphisms of $X$ are orientation preserving since $X$ is positive definite) and $BDiff(X)$ the classifying space. The tautological classes can be constructed for the universal bundle $U_X = EDiff(X) \times_{Diff(X)} X$, giving classes
\[
\kappa_c = \kappa_c(U_X) \in H^*( BDiff(X) ; \mathbb{Q}).
\]

The {\em tautological ring} of $X$
\[
R^*(X) \subseteq H^*( BDiff(X) ; \mathbb{Q})
\]
is defined as the subring of $H^*( BDiff(X) ; \mathbb{Q})$ generated by tautological classes $\kappa_c$, $c \in H^*(BSO(4) ; \mathbb{Q})$. Since $H^*( BSO(4) ; \mathbb{Q})$ is generated over $\mathbb{Q}$ by $p_1$ and $e$, it follows that $R^*(X)$ is generated by the classes $\{ \kappa_{ p_1^a e^b} \}_{a,b \ge 0}$. Similarly, for any family $E \to B$ we can define the tautological ring of $E$
\[
R^*(E) \subseteq H^*( B ; \mathbb{Q})
\]
to be the subring generated by the tautological classes $\kappa_c(E), c \in H^*(BSO(4) ; \mathbb{Q})$.

Tautological rings have been studied extensively for families of oriented surfaces, eg, \cite{mum}, \cite{mil}, \cite{loo}, \cite{fab}, \cite{mor} and there is a growing literature on tautological classes in higher dimensions \cite{grw}, \cite{gri}, \cite{ggrw}, \cite{rw}, \cite{rw}, \cite{bfj}, \cite{hllrw}. However as far as we are aware, our paper is the first to use gauge theory to obtain results on the tautological classes of $4$-manifolds. 

Let $B$ be a compact smooth manifold. A topological fibre bundle $E \to B$ with transition functions valued in $Diff(X)$ may be obtained by pullback of the universal family $U_X \to BDiff(X)$ with respect to a continuous map $B \to BDiff(X)$. As explained in \cite[\textsection 4.2]{bako}, it follows from a result of M\"uller--Wockel \cite{muwo} that such a family $E \to B$ admits a smooth structure for which $\pi$ is a submersion and the fibres of $E$ with their induced smooth structure are diffeomorphic to $X$. Since $E$ is smooth, we may use parametrised gauge theory to study the tautological classes $\kappa_c(E) \in H^*(B ; \mathbb{Q})$. If a relation amongst tautological classes holds in $R^*(E)$ for all compact, smooth families $\pi : E \to B$ with fibres diffeomorphic to $X$, then it must also hold in $R^*(X)$. This is because rational cohomology classes of $BDiff(X)$ are detected by continuous maps from compact, smooth manifolds into $BDiff(X)$. The upshot of this is that we can use gauge theory to indirectly study the tautological ring of $X$. 

\subsection{Main results}

Our first main result is the determination of the tautological rings of $\mathbb{CP}^2$ and $\mathbb{CP}^2 \# \mathbb{CP}^2$.

\begin{theorem}\label{thm:cp2}
The tautological rings of $\mathbb{CP}^2$ and $\mathbb{CP}^2 \# \mathbb{CP}^2$ are given by
\begin{itemize}
\item[(1)]{$R^*(\mathbb{CP}^2) = \mathbb{Q}[ \kappa_{p_1^2}, \kappa_{p_1^4}]$.}
\item[(2)]{$R^*(\mathbb{CP}^2 \# \mathbb{CP}^2 ) = \mathbb{Q}[ \kappa_{p_1^2}, \kappa_{p_1^3}]$.}
\end{itemize}
\end{theorem}

Variants $R^*(X , *)$ and $R^*(X , D^4)$ of the tautological ring are defined in \cite{gri} and \cite{ggrw}. Their definition is recalled in Section \ref{sec:n=1}. We determine these rings for $\mathbb{CP}^2$.
\begin{theorem}
We have ring isomorphisms
\begin{itemize}
\item[(1)]{$R^*(\mathbb{CP}^2 , * ) \cong \mathbb{Q}[ p_1 , e]$.}
\item[(2)]{$R^*(\mathbb{CP}^2 , D^4) \cong \mathbb{Q}$.}
\end{itemize}
\end{theorem}

The rings $R^*(\mathbb{CP}^2), R^*(\mathbb{CP}^2, *), R^*(\mathbb{CP}^2 , D^4)$ were investigated in \cite{rw}, but their structure was not fully determined. By computing these rings we have settled some open problems posed in \cite{rw}.

For each pair of non-negative integers $a,b$ we define a two variable polynomial $\phi_{a,b}(x,y) \in \mathbb{Z}[x,y]$ as follows. Let 
\[
p(z) = z^3-xz-y, \quad p'(z) = 3z^2-x.
\]
Then we define
\[
\phi_{a,b}(x,y) = \dfrac{1}{2\pi i}  \ointctrclockwise \dfrac{ (p'(z)+3x)^a  (p'(z))^b }{p(z)} dz
\]
where the contour encloses all zeros of $p(z)$. From this definition it follows that $\phi_{a,b}$ satisfies the following recursive formulas:
\begin{align*}
\phi_{a+1,b}(x,y) &= \phi_{a,b+1}(x,y) + 3 x \phi_{a,b}(x,y), \\
\phi_{a,b+3}(x,y) &= 3x \phi_{a,b+2}(x,y) + (27y^2-4x^3) \phi_{a,b}(x,y),
\end{align*}
which together with the initial conditions
\[
\phi_{0,0}(x,y) = 0, \quad \phi_{0,1}(x,y) = 3, \quad \phi_{0,2}(x,y) = 3x
\]
can be used to compute $\phi_{a,b}$ for all values of $a,b$. We will make use of the polynomials $\phi_{a,b}$ in the computation of tautological classes of families of definite $4$-manifolds. We first state the $n=1$ case.

\begin{theorem}\label{thm:n=1}
Let $E \to B$ be a smooth family with fibres diffeomorphic to $X$, where $X$ is a smooth, compact, simply-connected, positive definite $4$-manifold with $b_2(X)=1$. Suppose that the monodromy action of $\pi_1(B)$ on $H^2(X ; \mathbb{Z})$ is trivial. Then there exist classes $B \in H^4( B ; \mathbb{Q})$, $C \in H^6( B ; \mathbb{Q})$ such that:
\begin{itemize}
\item[(i)]{There is an isomorphism of $H^*(B ; \mathbb{Q})$-algebras
\[
H^*(E ; \mathbb{Q}) \cong H^*(B ; \mathbb{Q})[x]/( x^3 - Bx-C).
\]
}
\item[(ii)]{The Euler class and first Pontryagin class of $T(E/B)$ are given by
\[
e = 3x^2-B, \quad p_1 = 3x^2+2B.
\]
}
\item[(iii)]{For all $a,b \ge 0$, we have
\[
\kappa_{p_1^a e^b}(E) = \phi_{a,b}( B , C ).
\]
}
\end{itemize}
\end{theorem}

An interesting consequence of Theorem \ref{thm:n=1} is that the rational cohomology class $p_1$ depends only on the underlying topological structure of the family, because $p_1$ is completely determined by $x$ and $B$ and in turn these classes can be uniquely characterised in terms of the push-forward map $\pi_* : H^*(E ; \mathbb{Q}) \to H^{*-4}(B ; \mathbb{Q})$. In fact, a similar result is true more generally for families of definite $4$-manifolds. See Remark \ref{rem:topstruct}. It follows that the tautological classes $\kappa_{p_1^a e^b}(E)$ depend only on the underlying topological structure of the family. This could also be deduced from the fact that the tautological classes are also defined for topological bundles, see \cite[Theorem B]{erw}.

\begin{remark}\label{rem:cp2}
If $E = \mathbb{P}(V)$ is the $\mathbb{CP}^2$-bundle associated to a complex rank $3$ vector bundle $V \to B$ with trivial determinant, then $B = -c_2(V)$, $C = -c_3(V)$ and so $\kappa_{p_1^a e^b}(E) = \phi_{a,b}( -c_2(V) , -c_3(V) )$, which gives the tautological classes as polynomials in $c_2(V), c_3(V)$.
\end{remark}

To state our next result, we need a few definitions. Let $\Lambda_n$ denote a free abelian group of rank $n$ and let $e_1, \dots , e_n$ be a basis. Equip $\Lambda_n$ with the standard Euclidean inner product $\langle e_i , e_j \rangle = \delta_{i j}$. Let $W_n$ denote the isometry group of $\Lambda_n$. Then $W_n$ is isomorphic to a semidirect product
\[
W_n = S_n \ltimes \mathbb{Z}_2^n,
\]
where the symmetric group $S_n$ acts by permutation: $\sigma( e_i ) = e_{\sigma(i)}$ and the normal subgroup $\mathbb{Z}_2^n$ is generated by $\theta_1 , \dots , \theta_n$, where $\theta_i$ is the reflection in the hyperplane orthogonal to $e_i$. Let $X$ denote a smooth, compact, simply-connected $4$-manifold with positive definite intersection form and $n = b_2(X) \ge 1$. In Section \ref{sec:families} we construct a principal $W_n$-bundle
\[
p : \overline{BDiff}(X) \to BDiff(X)
\]
over $BDiff(X)$. Since $p$ is a finite covering, it follows that the pullback map $p^* : H^*( BDiff(X) ; \mathbb{Q}) \to H^*( \overline{BDiff}(X) ; \mathbb{Q})$ is injective and that the image is precisely the $W_n$-invariant part of $H^*(\overline{BDiff(X)} ; \mathbb{Q})$. In particular, we may think of the tautological ring $R^*(X)$ as a subring of $H^*( \overline{BDiff}(X) ; \mathbb{Q})$.

\begin{theorem}\label{thm:tautring}
Let $X$ be a smooth, compact, simply-connected, positive definite $4$-manifold with $b_2(X)=n \ge 2$. Then there exists classes
\[
D_{ij} \in H^2( \overline{BDiff}(X) ; \mathbb{Q}),  1 \le i,j  \le n, i \neq j
\]
with the following properties.
\begin{itemize}
\item[(i)]{Let $D^*(X) \subseteq H^*( \overline{BDiff}(X) ; \mathbb{Q})$ be the subring generated by $\{ D_{i,j} \}_{i \neq j}$. The group $W_n$ acts on the subring $D^*(X)$ according to
\[
\sigma( D_{ij} ) = D_{ \sigma(i) \sigma(j) }, \quad \sigma \in S_n
\]
and
\[
\theta_k ( D_{ij} ) = \begin{cases} D_{ij} & k \neq j, \\ -D_{ij} & k=j. \end{cases}
\]
}
\item[(ii)]{Let $I^*(X)$ denote the $W_n$-invariant subring of $D^*(X)$. Then $I^*(X) \subseteq H^*( BDiff(X) ; \mathbb{Q})$.}
\item[(iii)]{The tautological ring $R^*(X)$ of $X$ is a subring of $I^*(X)$. That is, all tautological classes of $X$ can be expressed as $W_n$-invariant polynomials of the $D_{ij}$.}
\end{itemize}
\end{theorem}

Theorem \ref{thm:tautring} says that the tautological classes can be written as $W_n$-invariant polynomials of the $D_{ij}$. The next theorem addresses the question of how to compute these invariant polynomials. First we set
\[
I_1 = \sum_{i,j | i \neq j} D_{ij}^2, \quad I_2 =  \sum_{i,j,k | i,j,k \text{ distinct}} D_{ik} D_{jk}.
\]
Then for $i = 1, \dots , n$, define
\[
B_i = \dfrac{3}{2} \sum_{j | j \neq i} D_{ij}^2 - \dfrac{(n-5)}{2n(n-1)} I_1 - \dfrac{1}{(n-1)} I_2,
\]
\[
C_i = \dfrac{1}{(n-1)} \sum_{j | j \neq i} \left( D_{ji}^3 - B_i D_{ji}\right).
\]

\begin{theorem}\label{thm:tautclasses}
Let $X$ be a smooth, compact, simply-connected, positive definite $4$-manifold with $b_2(X)=n \ge 2$. Then
\begin{itemize}
\item[(1)]{The tautological classes $\kappa_c$ with $c = p_1^a e^{2b}$ are given by 
\[
\kappa_{p_1^a e^{2b}} = \sum_{i=1}^n \phi_{a,2b}( B_i , C_i )
\]
}
\item[(2)]{The tautological classes $\kappa_c$ with $c = p_1^a e^{2b+1}$ are given by 
\[
\kappa_{p_1^a e^{2b+1}} = \sum_{i=1}^n \phi_{a,2b+1}(B_i,C_i) - 2 \sum_{j | j \neq i} (3D_{ij}^2+2B_j)^a( 3D_{ij}^2-B_j)^{2b}.
\]
}
\end{itemize}
\end{theorem}

Theorem \ref{thm:tautclasses} gives a completely explicit expression for the tautological classes $\kappa_c$ as polynomials in $\{ D_{ij} \}_{i \neq j}$, once the polynomials $\phi_{a,b}(x,y)$ are known. As an application of Theorem \ref{thm:tautclasses}, we prove the existence of many linear relations amongst tautological classes.

\begin{theorem}\label{thm:linrel}
Let $X$ be a smooth, compact simply-connected definite $4$-manifold and let $d \ge 1$ be given. Then amongst all tautological classes $\kappa_{p_1^a e^b}$ with $a+b=d$ and $b$ even, there are at least
\[
\lfloor d/2 \rfloor - \lfloor (d-1)/3 \rfloor
\]
linear relations. More precisely, if $c_0 , c_1 , \dots , c_{\lfloor d/2 \rfloor} \in \mathbb{Q}$ are such that
\begin{equation}\label{equ:lincomb}
\sum_{j=0}^{ \lfloor d/2 \rfloor } c_j \phi_{d-2j , 2j}(x,y) = 0,
\end{equation}
then we also have
\[
\sum_{j=0}^{\lfloor d/2 \rfloor} c_j \kappa_{p_1^{d-2j} e^{2j} } = 0
\]
and the space of $(c_0 , c_1 , \dots , c_{\lfloor d/2 \rfloor } )$ satisfying Equation (\ref{equ:lincomb}) has dimension at least
\[
\lfloor d/2 \rfloor - \lfloor (d-1)/3 \rfloor.
\]
\end{theorem}

As explained in Section \ref{sec:linrel}, for each $d \ge 2$, the families signature theorem gives one linear relation amongst the tautological classes $\kappa_{p_1^a e^b}$ with $a+b=d$ and $b$ even. Theorem \ref{thm:linrel} implies that there are further linear relations whenever $ \lfloor d/2 \rfloor - \lfloor (d-1)/3 \rfloor  > 1$. This is the case if $d=6$ or $d \ge 8$. The first few such relations (up to $d = 12$) are:
\begin{align*}
0 &= 4\kappa_{p_1^4 e^2} - 41 \kappa_{p_1^2 e^4} + 100 \kappa_{e^6}, \\
0 &= 36 \kappa_{p_1^6 e^2} - 461 \kappa_{p_1^4 e^4} + 1843 \kappa_{p_1^2 e^6} - 2300 \kappa_{e^8}, \\
0 &= 24 \kappa_{p_1^7 e^2} - 322 \kappa_{p_1^5 e^4} + 1379 \kappa_{p_1^3 e^6} - 1900 \kappa_{p_1 e^8}, \\
0 &= 108 \kappa_{p_1^8 e^2}-1579 \kappa_{p_1^6 e^4} + 7902 \kappa_{p_1^4 e^6} - 15531 \kappa_{p_1^2 e^8} + 9100 \kappa_{e^{10}}, \\
0 &= 360 \kappa_{p_1^9 e^2}-5606 \kappa_{p_1^7 e^4}+30923 \kappa_{p_1^5 e^6} -71311 \kappa_{p_1^3 e^8} + 57100 \kappa_{p_1 e^{10}}, \\
0 &= 144 \kappa_{p_1^8 e^4}-2552 \kappa_{p_1^6 e^6}+16629 \kappa_{p_1^4 e^8} - 47400 \kappa_{p_1^2 e^{10}} + 50000 \kappa_{e^{12}}, \\
0 &= 6000 \kappa_{p_1^{10} e^2} -98012 \kappa_{p_1^8 e^4} + 577796 \kappa_{p_1^6 e^6} -1461667 \kappa_{p_1^4 e^8} + 1338700 \kappa_{p_1^2 e^{10}}.
\end{align*}

\subsection{Idea behind main results}

The inspiration for our main results comes from considering Donaldson theory for a family of definite $4$-manifolds. Let $X$ be a compact, simply-connected smooth $4$-manifold with positive definite intersection form. Recall the proof of Donaldson's diagonalisation theorem uses the moduli space $\mathcal{M}$ of self-dual instantons on an $SU(2)$-bundle $E \to X$ with $c_2(E)=-1$. Then $\mathcal{M}$ is a $5$-dimensional oriented manifold with singularities. The singularities correspond to reducible instantons, which correspond to elements in $\xi \in H^2( X ; \mathbb{Z})$ satisfying $\xi^2=1$, considered modulo $\xi \mapsto -\xi$. Each singularity of $\mathcal{M}$ takes the form of a cone over $\mathbb{CP}^2$. The moduli space $\mathcal{M}$ is non-compact, but it admits a compactification $\overline{\mathcal{M}}$ whose boundary is diffeomorphic to $X$. Removing from $\overline{\mathcal{M}}$ a neighbourhood of each singularity, we obtain a cobordism $\mathcal{M}'$ from $X$ to a disjoint union of copies of $\mathbb{CP}^2$. Cobordism invariance of the signature implies that there are $n = b_2(X)$ copies of $\mathbb{CP}^2$ and hence, there are $n$ distinct pairs of elements $\pm \xi_1 , \dots , \pm \xi_n \in H^2( X ; \mathbb{Z})$ satisfying $\xi_i^2 = 1$. This implies that $H^2(X ; \mathbb{Z})$ is diagonalisable.

Now suppose that $E \to B$ is a smooth family with fibres diffeomorphic to $X$ and suppose for simplicity that the monodromy action of $\pi_1(B)$ on $H^2(X ; \mathbb{Z})$ is trivial. Considering the moduli space of self-dual instantons with $c_2(E) = -1$ on each fibre of $E$, we obtain a families moduli space $\mathcal{M}_E \to B$. Note that $\mathcal{M}_E$ is typically not a fibre bundle since the topology of the fibres of $\mathcal{M}_E$ can vary as we move in $B$. We would expect that for a sufficiently generic family of metrics on $E$, we can arrange that $\mathcal{M}_E$ is smooth away from reducible solutions and that the structure of $\mathcal{M}_E$ around the reducibles is given by taking fibrewise cones on $n$ $\mathbb{CP}^2$ bundles $E_1, \dots , E_n$ over $B$. We would further expect that $\mathcal{M}_E$ can be compactified by adding a boundary which is diffeomorphic to the family $E$. Removing a neigbourhood of the reducible solutions, we would expect to obtain a cobordism $\pi' : \mathcal{M}'_E \to B$ relative $B$, between $E \to B$ and the disjoint union of $\mathbb{CP}^2$-bundles $E_1, \dots , E_n$. Consider the virtual vector bundle $V = T\mathcal{M}'_E - (\pi')^*(TB)$. Clearly $V|_{E} = T(E/B)$ and $V|_{E_i} = T(E_i/B)$ for each $i$, hence the Pontryagin classes of $V$ restrict to the Pontryagin classes of $E$ and $E_i$ on the boundary. Applying Stokes' theorem, we expect to obtain a kind of ``diagonalisation theorem" for the tautological classes:
\begin{equation}\label{equ:diagonal}
\kappa_{p_1^a e^{2b} }(E) = \sum_{i=1}^n \kappa_{p_1^a e^{2b}}(E_i).
\end{equation}
Note that we need to take even powers of $e$ because $e$ is unstable whereas $e^2 = p_2$ is stable.

There are some technical challenges for carrying out this argument rigorously. Most notably, it seems difficult to arrange unobstructedness of the families moduli space around the reducible solutions. It is well known that this can be done for a single moduli space by choosing a sufficiently generic metric, but extending this to families appears challenging. Nevertheless, the intuition provided by Donaldson theory turns out to be essentially correct. Theorem \ref{thm:tautclasses} provides a rigorous version of Equation (\ref{equ:diagonal}), where $\phi_{a,b}(B_i , C_i)$ plays the role of $\kappa_{p_1^a , e^b}(E_i)$. Note by Remark \ref{rem:cp2}, that if $-B_i$, $-C_i$ are the Chern classes of a rank $3$ vector bundle $V_i \to B$ with trivial determinant, then $\phi_{a,b}(B_i , C_i) = \kappa_{p_1^a e^b}(E_i)$, where $E_i$ is the $\mathbb{CP}^2$-bundle $E_i= \mathbb{P}(V_i)$. The natural candidate for $V_i$ is the families index of the instanton deformation complex around the corresponding reducible, except that we only know this exists as a {\em virtual} vector bundle.

We will prove the main results using families Seiberg--Witten theory, or more precisely the Bauer--Furuta cohomotopy refinement of Seiberg--Witten theory. The main advantage of this approach is that it allows us to avoid various transversality issues that typically arise in the construction of moduli spaces. It is quite surprising that Seiberg--Witten theory works here. The issue is that the families Seiberg--Witten moduli space is compact and there is no obvious relation between the families moduli space and the family $E$. What happens instead is that Seiberg--Witten theory gives constraints on the topology of the families index associated to families of Dirac operators on $E$. We use this to indirectly obtain a series of constraints on the cohomology ring $H^*(E ; \mathbb{Q})$ of the family $E$ and in turn this gives constraints on the tautological classes.

\subsection{Outline of paper}
In \textsection \ref{sec:families} we establish some basic results concerning families of definite $4$-manifolds. In \textsection \ref{sec:fambf} we consider the Bauer--Furuta refinement of Seiberg--Witten theory for a family of definite $4$-manifolds. The main result is Theorem \ref{thm:fsw1}. The rest of the section is concerned with understanding some of the implications of this theorem. In \textsection \ref{sec:cohomring}, we study in great detail the structure of the cohomology rings $H^*(E ; \mathbb{Q})$ of families of definite $4$-manifolds and in \textsection \ref{sec:tautclasses} we prove our main results concerning the tautological classes of such families. \textsection \ref{sec:n=1} and \textsection \ref{sec:n=2} are concerned with the special cases of $\mathbb{CP}^2$ and $\mathbb{CP}^2 \# \mathbb{CP}^2$ and finally in \textsection \ref{sec:linrel} we study linear relations in the tautological rings of definite $4$-manifolds.\\

\noindent{\bf Acknowledgments.} We thank Oscar Randal-Williams for helpful comments on the paper and for suggesting a simpler proof of Lemma \ref{lem:splitting}. The author was financially supported by the Australian Research Council Discovery Project DP170101054.

\section*{List of Commonly Used Notation}

\begin{abbrv}
\item[$X$] Smooth, compact, simply-connected definite $4$-manifold (\textsection \ref{sec:families})
\item[$n$] $n = b_2(X)$, the second Betti number of $X$ (\textsection \ref{sec:families})
\item[$\Lambda_n$] The lattice $\mathbb{Z}^n$ with the Euclidean inner product (\textsection \ref{sec:families})
\item[$e_1, \dots , e_n$] The standard basis of $\Lambda_n$ (\textsection \ref{sec:families})
\item[$W_n$] The isometry group of $\Lambda_n$, $W_n \cong S_n \ltimes \mathbb{Z}_2^n$ (\textsection \ref{sec:families})
\item[$\xi_1, \dots , \xi_n$] A framing of $H^2(X ; \mathbb{Z})$ (\textsection \ref{sec:families})
\item[$\pi :E \to B$] Smooth family with fibres diffeomorphic to $X$ (\textsection \ref{sec:families})
\item[$Diff(X)$] Diffeomorphism group of $X$ (\textsection \ref{sec:families})
\item[$Diff_0(X)$] Diffeomorphisms of $X$ acting trivially on $H^2(X ; \mathbb{Z})$ (\textsection \ref{sec:families})
\item[$K(X)$] The image of $Diff(X)$ in $Aut( H^2(X ; \mathbb{Z}))$ (\textsection \ref{sec:families})
\item[$BDiff(X), BDiff_0(X)$] The classifying spaces of $Diff(X), Diff_0(X)$ (\textsection \ref{sec:families})
\item[$\overline{BDiff}(X)$] $\overline{BDiff}(X) = BDiff_0(X) \times_{K(X)} W_n$ (\textsection \ref{sec:families})
\item[$\kappa_{p_1^a e^b} \in H^*( BDiff(X) ; \mathbb{Z})$] Tautological classes of $X$ (\textsection \ref{sec:intro}) 
\item[$R^*(X)$] Tautological ring of $X$ (\textsection \ref{sec:intro})
\item[$\kappa_{p_1^a e^b}(E) \in H^*(B ; \mathbb{Q})$] Tautological classes of $X$ evaluated on $E$ (\textsection \ref{sec:intro})
\item[$p_1, e \in H^4(E ; \mathbb{Q})$] First Pointryagin class and Euler class of the vertical tangent bundle of $E \to B$ (\textsection \ref{sec:families})
\item[$x_i \in H^2(E ; \mathbb{Q})$] Degree $2$ classes determined by Proposition \ref{prop:classes} (\textsection \ref{sec:families})
\item[$\nu \in H^4(E ; \mathbb{Q})$] Degree $4$ class determined by Proposition \ref{prop:classes} (\textsection \ref{sec:families})
\item[$D_{ij}^k \in H^2(B ; \mathbb{Q}), \linebreak E_{ij}, G_i \in H^4(B ; \mathbb{Q}), \linebreak J_i \in H^6(B ; \mathbb{Q}), \linebreak \omega \in H^8(B ; \mathbb{Q})$] Cohomology classes in $H^*(B ; \mathbb{Q})$ appearing as structure constants for the cup product on $H^*(E ; \mathbb{Q})$, see Equations (\ref{equ:21}), (\ref{equ:22}), (\ref{equ:23}), (\ref{equ:24}) (\textsection \ref{sec:families})
\item[ ] 
\item[$D_{ij}$] $D_{ij} = D_{ij}^i$, see Proposition \ref{prop:ring} (\textsection \ref{sec:cohomring})
\item[$B_i \in H^4(B ; \mathbb{Q})$] Class defined in Proposition \ref{prop:xcubed} (\textsection \ref{sec:cohomring})
\item[$C_i \in H^6(B ; \mathbb{Q})$] Class defined in Proposition \ref{prop:xcubed} (\textsection \ref{sec:cohomring})
\item[$\mu \in H^4(B ; \mathbb{Q})$] $p_1 = 3(x_1^2 + \cdots + x_n^2) + \mu$, Proposition \ref{prop:p1} (\textsection \ref{sec:cohomring})

\end{abbrv}

\section{Families of definite $4$-manifolds}\label{sec:families}

Throughout the paper $X$, denotes a smooth, compact, simply-connected $4$-manifold with positive definite intersection form and $n = b_2(X) \ge 1$. The intersection pairing $\langle \; , \; \rangle$ on $H^2(X)$ is a symmetric, unimodular bilinear form. By Donaldson's diagonalisation theorem \cite{don}, the intersection form on $H^2(X ; \mathbb{Z})$ is diagonal and so there exists an orthonormal basis $\xi_1, \dots , \xi_n$ for $H^2(X ; \mathbb{Z})$. An orthonormal basis for $H^2(X ; \mathbb{Z})$ will be called a {\em framing} of $H^2(X ; \mathbb{Z})$. Let $\Lambda_n$ denote the free abelian group $\mathbb{Z}^n$ of rank $n$ and let $e_1, \dots , e_n$ be the standard basis. Equip $\Lambda_n$ with the standard Euclidean inner product. Then a framing $\xi_1, \dots , \xi_n$ of $H^2(X ; \mathbb{Z})$ determines an isometry $\phi : \Lambda_n \to H^2(X ; \mathbb{Z})$ given by $\phi(e_i) = \xi_i$. Let $W_n = Aut( \Lambda_n)$ denote the symmetry group of $\Lambda_n $ equipped with its intersection form. Since the only classes of norm $1$ are $\pm e_1, \dots, \pm e_n$, it is easy to see that $W_n$ is isomorphic to a semidirect product
\[
W_n = S_n \ltimes \mathbb{Z}_2^n,
\]
where the symmetric group $S_n$ acts by permutation: $\sigma( e_i ) = e_{\sigma(i)}$ and the normal subgroup $\mathbb{Z}_2^n$ is generated by $\theta_1 , \dots , \theta_n$, where $\theta_i$ is the reflection in the hyperplane orthogonal to $e_i$:
\[
\theta_i( e_j ) = \begin{cases} e_j & j \neq i, \\ -e_j & j = i. \end{cases}
\]
$W_n$ is also the Weyl group of the root systems $B_n$ and $C_n$. Clearly $W_n$ is a subgroup of the isometry group of $\Lambda_n$. For the reverse inclusion, note that any isometry must permute the vectors of unit length which are $\pm e_1, \pm e_2, \dots , \pm e_n$. Hence any isometry of $\Lambda_n$ is given by a permutation of $e_1, \dots , e_n$, followed by some sign changes $e_i \mapsto -e_i$.

Let $Diff(X)$ denote the group of orientation preserving diffeomorphisms of $X$ with the $\mathcal{C}^\infty$-topology and $Diff_0(X)$ the subgroup acting trivially on $H^2(X ; \mathbb{Z})$. Equivalently, $Diff_0(X)$ is the subgroup of $Diff(X)$ preserving a framing of $H^2(X ; \mathbb{Z})$. By definition we have a short exact sequence
\[
1 \to Diff_0(X) \to Diff(X) \to K(X) \to 1
\]
where $K(X)$ is the image of the map $Diff(X) \to Aut( H^2( X ; \mathbb{Z}) )$, which sends a diffeomorphism $f : X \to X$ to the induced map $(f^{-1})^* : H^2(X ; \mathbb{Z}) \to H^2(X ; \mathbb{Z})$. Note note that if $X = \#^n \mathbb{CP}^2$, then $K(X)=Aut( H^2( X ; \mathbb{Z} )$. One sees this as follows. There is an orientation preserving diffeomorphism of $\mathbb{CP}^2$ which acts as $-1$ on $H^2( \mathbb{CP}^2 ; \mathbb{Z})$, namely complex conjugation. Such a diffeomorphism can be isotopied so as to act as the identity on a disc in $\mathbb{CP}^2$, hence can be extended to the connected sum $\#^n \mathbb{CP}^2$. Since we can do this for each summand, we see that $\theta_1 , \dots , \theta_n \in K(X)$. To see that $S_n \subset K(X)$, regard $X$ as $S^4$ with $\mathbb{CP}^2$ attached at $n$ points. Since these $n$ points can be permuted by diffeomorphisms of $S^4$, it follows that $S_n \subset K(X)$.

Fixing a framing $\xi_1, \dots , \xi_n$ of $H^2(X ; \mathbb{Z})$, we can identify $K(X)$ with a subgroup of $W_n$. Since $W_n$ is finite, so is $K(X)$. Taking classifying spaces, we see that $BDiff_0(X)$ has the structure of a principal $K(X)$-bundle over $BDiff(X)$. We now define
\[
\overline{BDiff}(X) = BDiff_0(X) \times_{K(X)} W_n.
\]
So $\overline{BDiff}(X)$ is a principal $W_n$-bundle over $BDiff(X)$. Let $p : \overline{BDiff}(X) \to BDiff(X)$ be the covering map. Since $p$ is a finite covering, it follows that the pullback map $p^* : H^*( BDiff(X) ; \mathbb{Q}) \to H^*( \overline{BDiff}(X) ; \mathbb{Q})$ is injective and that the image is precisely the $W_n$-invariant part of $H^*(\overline{BDiff}(X) ; \mathbb{Q})$. Therefore we may identify the tautological ring $R^*(X)$ with a subring of $H^*( \overline{BDiff}(X) ; \mathbb{Q})$
\[
R^*(X) \subseteq H^*(BDiff(X) ; \mathbb{Q} ) \subseteq H^*(\overline{BDiff}(X) ; \mathbb{Q}).
\]

\begin{remark}
Since $\overline{BDiff}(X) = BDiff_0(X) \times_{K(X)} W_n$, we have a fibration $BDiff_0(X) \to \overline{BDiff}(X) \to W_n/K(X)$. But $W_n/K(X)$ is a finite discrete set, so $\overline{BDiff}(X)$ is just the disjoint union of $|W_n/K(X)|$ copies of $BDiff_0(X)$. For this reason it makes little difference whether we work with $BDiff_0(X)$ or $\overline{BDiff}(X)$. We prefer to use $\overline{BDiff}(X)$ because the whole isometry group $W_n$ acts on this space. Note also that for $X = \#^n \mathbb{CP}^2$, we have $K(X) = W_n$ and so $\overline{BDiff}(X) = BDiff_0(X)$ in this case.
\end{remark}

Let $\pi : E \to B$ be a family with fibres diffeomorphic to $X$. Then $E$ admits a reduction of structure to $Diff_0(X)$ if and only if the monodromy action of $\pi_1(B)$ on $H^2$ of the fibres is trivial. In such a case, if we choose a framing $\xi_1, \dots , \xi_n$ of a single fibre and parallel translate, we obtain a framing $\xi_1(b) , \dots , \xi_n(b)$ of $H^2(X_b ; \mathbb{Z})$ for each $b \in B$ such that the framing varies continuously with $b$. Henceforth we will restrict attention to families $\pi : E \to B$ equipped with a reduction of structure group to $Diff_0(X)$. We assume further that a framing has been chosen.

\begin{proposition}\label{prop:degen}
Let $\pi : E \to B$ be a family with structure group $Diff_0(X)$. Then the Leray-Serre spectral sequence for $H^*( E ; \mathbb{Q})$ degenerates at $E_2$.
\end{proposition}
\begin{proof}
It suffices prove the result when $B$ is connected. Let $e \in H^4( E ; \mathbb{Q})$ denote the Euler class of the vertical tangent bundle. For each $b \in B$, we have that $e|_{X_b}$ is $(2+n)$ times a generator of $H^4(X_b ; \mathbb{Q})$. It follows that all the differentials of the form $d_r : E_r^{0,4} \to E_r^{r,5-r}$ are zero. Moreover the differentials for $r$ odd are all zero because $H^*(X ; \mathbb{Q})$ is non-zero only in even degrees. Next, note that $E_2^{0,2} \cong H^2( X ; \mathbb{Q})$ (since $B$ is connected). Thus we can identify $\xi_1, \dots , \xi_n$ with classes in $E_2^{0,2}$. Now $\xi_j^2 \in E_2^{0,4}$, so 
\[
0 = d_3( \xi_j^2) = 2 \xi_j d_3(\xi_j) \in E_3^{3,2} \cong H^2(X ; \mathbb{Q}) \otimes H^3(B ; \mathbb{Q}).
\]
Hence $d_3(\xi_j) = 0$. It follows that there exist classes $x_1, \dots x_n \in H^2( E ; \mathbb{Q})$ such that $x_j|_{X_b} = \xi_j(b)$. Now the result follows by the Leray-Hirsch theorem.
\end{proof}

As seen in the proof of Proposition \ref{prop:degen}, there exist classes $x_1, \dots x_n \in H^2(E ; \mathbb{Q})$ such that $x_j|_{X_b} = \xi_j(b)$ (note that in general the classes $x_j$ can't be taken to lie in $H^2(E ; \mathbb{Z})$). The $x_i$ are not unique because if $a \in H^2(B ; \mathbb{Q})$ then $x_j + \pi^*(a)$ also restricts to $\xi_j(b)$ on $X_b$. From the Leray-Serre spectral sequence it is clear that the $x_j$ are unique up to such shifts.\\

Let 
\[
\int_{E/B} : H^k( E ; \mathbb{Q}) \to H^{k-4}( B ; \mathbb{Q})
\]
denote fibre integration. We clearly have
\[
\int_{E/B} x_j = 0, \quad \int_{E/B} x_j^2 = 1, \quad \int_{E/B} x_i x_j = 0
\]
for all $i,j$ with $j \neq i$. Let $e \in H^4( E ; \mathbb{Q})$ denote the Euler class and $p_j \in H^{4j}( E ; \mathbb{Q})$ the Pontryagin classes of the vertical tangent bundle. Since the fibres are $4$-dimensional, we have $p_j = 0$ for $j > 2$ and $p_2 = e^2$. So all rational characteristic classes of the vertical tangent bundle can be expressed in terms of $p_1$ and $e$. From the Gauss-Bonnet and signature theorems, we have
\[
\int_{E/B} e = \chi(X) = n+2, \quad \int_{E/B} p_1 = 3\sigma(X) = 3n.
\]

\begin{proposition}\label{prop:classes}
Let $\pi : E \to B$ be a family with structure group $Diff_0(X)$ and framing $\xi_1, \dots, \xi_n$. Then there exist uniquely determined classes $x_1, \dots x_n \in H^2(E ; \mathbb{Q})$ and $\nu \in H^4(E ; \mathbb{Q})$ such that
\begin{itemize}
\item[(1)]{$x_j |_{X_b} = \xi_j(b)$ for $j=1,\dots n$,}
\item[(2)]{$\int_{E/B} x_j^3 = 0$ for $j=1,\dots n$,}
\item[(3)]{$\int_{E/B} \nu = 1$,}
\item[(4)]{$\int_{E/B} x_j \nu = 0$ for $j=1,\dots n$,}
\item[(5)]{$\int_{E/B} \nu^2 =0$.}
\end{itemize}
\end{proposition}
\begin{proof}
We already saw that there exist classes $y_1, \dots , y_n \in H^2( E ; \mathbb{Q})$ such that $y_j |_{X_b} = \xi_j(b)$. Now set
\[
x_j = y_j - \dfrac{1}{3} \pi^*\left( \int_{E/B} y_j^3 \right).
\]
Then it is straightforward that the $x_j$ satisfy (1) and (2). Now let $\nu_0 = x_1^2 \in H^4( E ; \mathbb{Q})$. This satisfies (3). Now set
\[
\nu_1 = \nu_0 - \sum_{j=1}^n \pi^*\left( \int_{E/B} x_j \nu_0\right) x_j.
\]
Then $\nu_1$ clearly satisfies (3) and (4).  Moreover any other class satisfying (3) and (4) must be of the form $\nu_1 + \pi^*( a)$ for some $a \in H^4(B ; \mathbb{Q})$. Set
\[
\nu = \nu_1 - \dfrac{1}{2} \pi^* \left( \int_{E/B} \nu_1^2 \right).
\]
Then $\nu$ satisfies (3),(4),(5). Uniqueness of $\nu$ and the $x_i$ is straightforward.
\end{proof}

In summary, $H^*(E ; \mathbb{Q})$ is a free $H^*(B ; \mathbb{Q})$-module with a uniquely determined basis $1,x_1, \dots , x_n , \nu$ satisfying
\begin{itemize}
\item[(1)]{ $\int_{E/B} 1 = 0$,}
\item[(2)]{ $\int_{E/B} x_j = 0$ for $j=1,\dots n$,}
\item[(3)]{ $\int_{E/B} x_j^2 = 1$ for $j=1,\dots n$,}
\item[(4)]{ $\int_{E/B} x_i x_j = 0$ for $i,j=1,\dots n$, $i \neq j$,}
\item[(5)]{$\int_{E/B} x_j^3 = 0$ for $j=1,\dots n$,}
\item[(6)]{$\int_{E/B} \nu = 1$,}
\item[(7)]{$\int_{E/B} x_j \nu = 0$ for $j=1,\dots n$,}
\item[(8)]{$\int_{E/B} \nu^2 =0$.}
\end{itemize}

The cup product on $H^*(E ; \mathbb{Q})$ will be completely determined by the products
\[
x_i x_j \text{ for } i \neq j, \quad x_i^2, \quad x_i \nu, \quad \nu^2.
\]
By (1)-(8) above, these products must have the form
\begin{equation*}
\begin{aligned}
x_i x_j &= \sum_{k} D_{ij}^k x_k + E_{ij}, \\
x_i^2 & = \nu + \sum_{j} F_{ij} x_j + G_i \\
x_i \nu &= \sum_{j} I_{ij} x_j + J_i \\
\nu^2 &= \sum_j K_j x_j + \omega
\end{aligned}
\end{equation*}
for some classes $D_{ij}^k, F_{ij} \in H^2(B ; \mathbb{Q})$, $E_{ij}, G_i, I_{ij} \in H^4(B ; \mathbb{Q})$, $J_i, K_i \in H^6(B ; \mathbb{Q})$, $\omega \in H^8(B ; \mathbb{Q})$. We can assume also that $D_{ij}^k$ is symmetric in $i$ and $j$. Note that the classes $D_{ij}^k , \dots , K_i$ are uniquely determined because $1,x_1, \dots , x_n , \nu$ is a basis for $H^*(E ; \mathbb{Q})$ as a $H^*(B ; \mathbb{Q})$-module.\\

\begin{proposition}\label{prop:ident1}
We have the following identities:
\begin{equation*}
\begin{aligned}
F_{ij} &= D_{ij}^i \text{ for } i \neq j, \\
F_{ii} &= 0, \\
I_{ij} &= E_{ij} \text{ for } i \neq j, \\
I_{ii} &= G_i, \\
K_i &= J_i.
\end{aligned}
\end{equation*}
\begin{proof}
We have 
\[
\int_{E/B} x_i^2 x_j = \int_{E/B} x_i \left( \sum_{k} D_{ij}^k x_k + E_{ij} \right) = D_{ij}^i.
\]
On the other hand
\[
\int_{E/B} x_i^2 x_j = \int_{E/B} x_j \left(  \nu + \sum_{k} F_{ik} x_k + G_i \right) = F_{ij}.
\]
Equating these gives $F_{ij} = D_{ij}^i$ for $i \neq j$. Similarly from $\int_{E/B} x_i^3=0$ we get that $F_{ii}=0$. Evaluating $\int_{E/B} x_i x_j \nu$ two different ways gives $I_{ij} = E_{ij}$ for $i \neq j$, evaluating $\int_{E/B} x_i^2 \nu$ in two different ways gives $I_{ii} = G_i$ and evaluating $\int_{E/B} x_i \nu^2$ in two different ways gives $K_i = J_i$.
\end{proof}
\end{proposition}

After making the simplifications given by Proposition \ref{prop:ident1} we have
\begin{align}
x_i x_j &= \sum_{k} D_{ij}^k x_k + E_{ij}, \label{equ:21} \\
x_i^2 & = \nu + \sum_{j | j\neq i} D_{ij}^i x_j + G_i \label{equ:22} \\
x_i \nu &= G_i x_i + \sum_{j | j \neq i} E_{ij} x_j + J_i \label{equ:23} \\
\nu^2 &= \sum_j J_j x_j + \omega. \label{equ:24}
\end{align}


\section{Families Bauer--Furuta Theory}\label{sec:fambf}

Let $X$ be a compact, oriented, smooth $4$-manifold with $b_1(X)=0$. Let $\mathfrak{s}$ be a spin$^c$-structure on $X$ with characteristic $c = c_1(\mathfrak{s}) \in H^2(X ; \mathbb{Z})$. Let $d = (c^2 - \sigma(X))/8$ be the index of the associated spin$^c$ Dirac operator.

Let $S^1$ act on $\mathbb{C}$ by scalar multiplication and trivially on $\mathbb{R}$. As shown by Bauer and Furuta \cite{bafu}, one can take a finite dimensional approximation of the Seiberg--Witten equations for $(X,\mathfrak{s})$ to obtain an $S^1$-equivariant map
\[
f : (\mathbb{C}^{a} \oplus \mathbb{R}^b)^+ \to (\mathbb{C}^{a'} \oplus \mathbb{R}^{b'})^+
\]
for some $a,b,a',b' \ge 0$, where $a-a' = d$, $b'-b = b_+(X)$. Here $T^+$ denotes the one-point compactification of $T$. By construction, $f$ sends the point at infinity in $(\mathbb{C}^{a} \oplus \mathbb{R}^b)^+$ to the point at infinity in $(\mathbb{C}^{a'} \oplus \mathbb{R}^{b'})^+$. Additionally, $f$ can be chosen so that its restriction $f|_{(\mathbb{R}^b)^+} : (\mathbb{R}^b)^+ \to (\mathbb{R}^{b'})^+$ is the map induced by an inclusion of vector spaces $\mathbb{R}^b \subseteq \mathbb{R}^{b'}$. For the purposes of this paper, it is more convenient to look at the Seiberg--Witten equations on $X$ {\em with the opposite orientation}. Once again we obtain an $S^1$-equivariant map of the form
\begin{equation}\label{equ:f}
f : (\mathbb{C}^{a} \oplus \mathbb{R}^b)^+ \to (\mathbb{C}^{a'} \oplus \mathbb{R}^{b'})^+,
\end{equation}
but now $a,a',b,b'$ satisfy $a'-a = d$, $b'-b = b_-(X)$.

The process of taking a finite dimensional approximation of the Seiberg--Witten equations can be carried out in families \cite{szy}, \cite{bako}. Let $B$ be a compact smooth manifold. Consider a smooth family $\pi : E \to B$ with fibres diffeomorphic to $X$ and suppose that there is a spin$^c$-structure $\mathfrak{s}_{E/B}$ on $T(E/B)$ which restricts to $\mathfrak{s}$ on the fibres of $E$. Taking a finite dimensional approximation of the Seiberg--Witten equations for the family $E$ (with the opposite orientation on $X$), we obtain a family of maps of the form (\ref{equ:f}). More precisely, we obtain complex vector bundles $V,V'$ over $B$ of ranks $a,a'$, real vector bundles $U,U'$ over $B$ of ranks $b,b'$ and an $S^1$-equivariant map of sphere bundles
\[
f : S_{V , U} \to S_{V' , U'}
\]
covering the identity on $B$. Here $S_{V,U}, S_{V',U'}$ denote the fibrewise one-point compactifications of $V \oplus U$ and $V' \oplus U'$. The group $S^1$ acts on $V,V'$ by scalar multiplication and trivially on $U,U'$. The action of $S^1$ on the direct sums $V \oplus U$, $V' \oplus U'$ extends continuously to the fibrewise one-point compactifications $S_{V,U}, S_{V',U'}$. Moreover, the following relations hold in $K^0(B)$ and $KO^0(B)$ respectively:
\[
V' - V = D, \quad \quad U' - U = H^-(X),
\]
where $D \in K^0(B)$ is the families index of the family of spin$^c$ Dirac operators on $E$ determined by $\mathfrak{s}_{E/B}$ and $H^-(X)$ is the vector bundle on $B$ whose fibre over $b \in B$ is the space of harmonic anti-self-dual $2$-forms on the fibre of $E$ over $b$ (with respect to some smoothly varying fibrewise metric on $E$). By stabilising the map $f$, we can assume that $V,U$ are trivial vector bundles. As shown in \cite{bako}, the map $f$ may be constructed so as to satisfy two further properties. First, we may assume that $U' \cong U \oplus H^-(X)$ and that the restriction $f|_{S_U} : S_U \to S_{U'}$ is the map induced by the inclusion $U \to U'$. Second, we may assume that $f$ sends the point at infinity in each fibre of $S_{V,U}$ to the point at infinity of the corresponding fibre of $S_{V',U'}$. Let $B_{V,U} \subseteq S_{V,U}$ denote the section at infinity and similarly define $B_{V',U'} \subseteq S_{V',U'}$. Then $f$ sends $B_{V,U}$ to $B_{V',U'}$. Hence $f$ defines an $S^1$-equivariant map of pairs
\[
f : (S_{V,U} , B_{V,U} ) \to (S_{V',U'} , B_{V',U'}).
\]

\begin{theorem}\label{thm:fsw1}
Suppose that $\pi : E \to B$ is a smooth family of simply-connected, positive definite $4$-manifolds over a compact base $B$ and that $T(E/B)$ admits a spin$^c$-structure $\mathfrak{s}_{E/B}$. Let $D \in K^0( B)$ denote the index of the family of spin$^c$-Dirac operators associated to $\mathfrak{s}_{E/B}$. Then $c_j(D) = 0$ for $j > d$, where $d$ is the virtual rank of $D$.
\end{theorem}
\begin{proof}
This result is a variant of \cite[Theorem 1.1]{bar}. We give a streamlined proof. As explained above, taking a finite dimensional approximation of the Seiberg--Witten equations for the family $E$ (with opposite orientation on $X$), we obtain an $S^1$-equivariant monopole map $f : S_{V,U} \to S_{V',U'}$. Since $X$ is positive definite, $H^-(X) = 0$ and $U' = U$. So $f$ takes the form
\[
f : S_{V,U} \to S_{V',U},
\]
with the property that $f|_{S_U}$ is the identity $S_U \to S_U$. We also have that $V' - V = D$. Let $\tau_{V,U}$ and $\tau_{V',U}$ denote the $S^1$-equivariant Thom classes of $S_{V,U}$ and $S_{V',U}$. Consider the commutative diagram
\[
\xymatrix{
(S_{V,U} , B_{V,U}) \ar[r]^-f & (S_{V',U} , B_{V,U}) \\
(S_U , B_U) \ar[u]^-j \ar[r]^-{id} & (S_U , B_U) \ar[u]^-{j'}
}
\]
By the Thom isomorphism in equivariant cohomology, we must have
\[
f^*(\tau_{V',U}) = \beta \, \tau_{V,U}
\]
for some $\beta \in H^{2d}_{S^1}( B ; \mathbb{Z})$. On the other hand $j^*(\tau_{V,U}) = e_{S^1}(V) \tau_U$ and $(j')^*(\tau_{V',U}) = e_{S^1}(V')\tau_U$, where $e_{S^1}(V), e_{S^1}(V')$ denote the $S^1$-equivariant Euler classes of $V$ and $V'$ and $\tau_U$ is the $S^1$-equivariant Thom class of $U$. Therefore
\begin{align*}
\beta e_{S^1}(V) \tau_U &= j^*( \beta \tau_{V,U} ) \\
&= j^* f^*( \tau_{V',U})  \\
&= (j')^*( \tau_{V',U}) \\
&= e_{S^1}(V') \tau_U.
\end{align*}
Hence
\begin{equation}\label{equ:eebeta}
e_{S^1}(V') = \beta e_{S^1}(V)
\end{equation}
for some $\beta \in H^{2d}_{S^1}(B ; \mathbb{Z})$. Note that since $S^1$ acts trivially on $B$ we have $H^*_{S^1}(B ; \mathbb{Z}) \cong H^*( B ; \mathbb{Z})[x]$, where $H^*_{S^1}( pt ; \mathbb{Z}) = \mathbb{Z}[x]$. Using a splitting principle argument, it is easy to see that if $S^1$ acts on a complex rank $m$ vector bundle $W$ by scalar multiplication, then
\[
e_{S^1}(W) = x^m + x^{m-1}c_1(W) + \cdots + c_m(W).
\]
Now by stabilisation, we may assume that $V$ is a trivial bundle: $V \cong \mathbb{C}^a$. Then $V'$ has the same Chern classes as $D$. So
\[
e_{S^1}(V') = x^{a'} + x^{a'-1}c_1(D) + \cdots + c_{a'}(D), \quad e_{S^1}(V) = x^a.
\]
Then, writing
\[
\beta = \beta_0 x^d + \beta_1 x^{d-1} + \cdots + \beta_d,
\] 
Equation (\ref{equ:eebeta}) becomes
\[
x^{a'} + x^{a'-1}c_1(D) + \cdots + c_{a'}(D) = \beta_0 x^{d+a} + \beta_1 x^{d+a-1} + \cdots + \beta_d x^a.
\]
Then since $d+a = a'$, it follows that $\beta_j = c_j(D)$ for $0 \le j \le d$ and that $c_j(D) = 0$ for $j > d$.
\end{proof}

\begin{lemma}\label{lem:splitting}
Let $P \to B$ be a principal $PU(m)$-bundle and $\pi : E \to B$ the associated $\mathbb{CP}^{m-1}$-bundle. Then the pullback $\pi^*(P)$ of $P$ to the total space of $E$ admits a lift of structure group to $U(m)$.
\end{lemma}
\begin{proof}
Let $G \subset PU(m)$ be the subgroup of $PU(m)$ fixing a point in $\mathbb{CP}^{m-1}$. Then clearly $\pi^*(P)$ admits a reduction of structure to $G$. On the other hand it is easy to see that $G \simeq U(m-1)$ and that the inclusion $G \to PU(m)$ factors through the projection $U(m) \to PU(m)$. Therefore $\pi^*(P)$ admits a lift of structure group to $U(m)$.
\end{proof}

\begin{lemma}\label{lem:splitting2}
There exists a fibre bundle $\rho : F \to B$ such that :
\begin{itemize}
\item[(1)]{$\rho^* : H^*(B ; \mathbb{Q}) \to H^*(F ; \mathbb{Q})$ is injective.}
\item[(2)]{For $i = 1, \dots , n$ there exist classes $\zeta_i \in H^2( \rho^*(E) ; \mathbb{Z})$ such that $\zeta_i$ restricted to the fibres of $\rho^*(E)$ equals $\xi_i$.}
\end{itemize}
\end{lemma}

The point of this lemma is that the $\zeta_i$ are {\em integral} cohomology classes whereas the $x_i$ defined earlier are only rational.

\begin{proof}
Consider the Leray-Serre spectral sequence $E^{p,q}_r$ for $\pi : E \to B$. Note that $E^{p,q}_2 = E^{p,q}_3$. We have seen that $d_3(\xi_j)$ is zero rationally, but it need not be zero over $\mathbb{Z}$. Therefore $g_j = d_3(\xi_j) \in E_3^{3,0} = H^3(B ; \mathbb{Z})$ for $j=1, \dots , n$ are all torsion classes. By a result of Serre \cite{gro}, every torsion class in $H^3(B ; \mathbb{Z})$ is represented by the lifting obstruction for some principal $PU(m)$-bundle, where the rank $m$ is allowed to vary. Thus for $i = 1, \dots , n$ we can find an $m_i$ and a principal $PU(m_i)$-bundle $P_i \to B$ such that $g_i$ is the lifting obstruction for $P_i$. Let $\pi_i : E_i \to B$ be the associated $\mathbb{CP}^{m_i-1}$-bundle. By Lemma \ref{lem:splitting}, the pullback of $g_i$ to $E_i$ must vanish.

Let $F = E_1 \times_B E_2 \times_B \cdots \times_B E_n$ and let $\rho : F \to B$ be the projection. By induction on $n$ it is straightforward to see that $\rho^*  : H^*(B ; \mathbb{Q}) \to H^*(F ; \mathbb{Q})$ is injective. Moreover we have $\rho^*(g_i) = 0$ for all $i$. Hence the Leray-Serre spectral sequence for $\rho^*(E) \to B$ degenerates over $\mathbb{Z}$ at $E_2$. Hence for $i = 1, \dots , n$, there exist classes $\zeta_i \in H^2( \rho^*(E) ; \mathbb{Z})$ such that $\zeta_i$ restricted to the fibres of $\rho^*(E)$ equals $\xi_i$.
\end{proof}

\begin{theorem}\label{thm:sw}
Let $\pi : E \to B$ be a family with structure group $Diff_0(X)$. Then
\[
\int_{E/B} e^{( \epsilon_1 x_1 + \cdots + \epsilon_n x_n)/2} \hat{A}(T(E/B)) = 0
\]
for all $\epsilon_1 , \dots , \epsilon_n \in \{1, -1\}$.
\end{theorem}

\begin{proof}
Let $\rho : F \to B$ be as in the statement of Lemma \ref{lem:splitting2}. Since $\rho^* : H^*(B ; \mathbb{Q}) \to H^*(F ; \mathbb{Q})$ is injective, to show that
\[
\int_{E/B} e^{( \epsilon_1 x_1 + \cdots + \epsilon_n x_n)/2} \hat{A}(T(E/B))
\]
is zero, it suffices to show that it pulls back to zero under $\rho$. Therefore we may restrict to families with the property that there exists classes $\zeta_1, \dots , \zeta_n \in H^2( E ; \mathbb{Z})$ such that $\zeta_i$ restricted to the fibres of $E$ equals $\xi_i$. Let
\[
c = \epsilon_1 \zeta_1 + \epsilon_1 \zeta_2 + \cdots + \epsilon_n \zeta_n \in H^2( E ; \mathbb{Z}).
\]
Then $c$ is a characteristic for $T(E/B)$ in the sense that the mod $2$ reduction of $c$ is $w_2( T(E/B))$. Therefore the third integral Stiefel-Whitney class of $T(E/B)$ vanishes and so $T(E/B)$ admits some spin$^c$-structure $\mathfrak{s}'$. Let $c' = c_1( \mathfrak{s}') \in H^2( E ; \mathbb{Z})$. Then since $c'$ is a characteristic for $T(E/B)$, we must have
\[
c' = \sum_{i=1}^n k_i \zeta_i + \pi^*(\eta)
\]
for some odd integers $k_1, \dots , k_n$ and some $\eta \in H^2(B ; \mathbb{Z})$. For each $i$, let $L_i \to E$ be the line bundle with $c_1(L_i) = \zeta_i$. The set of spin$^c$-structures for $T(E/B)$ is a torsor over the group of line bundles on $E$. So we may consider the spin$^c$-structure
\[
\mathfrak{s} = L_1^{a_1} \otimes L_2^{a_2} \otimes \cdots \otimes L_n^{a_n} \otimes \mathfrak{s'},
\]
where $a_i = (\epsilon_i-k_i)/2$. It follows that $c_1(\mathfrak{s}) = c + \pi^*(\eta)$. Now we apply Theorem \ref{thm:fsw1} to the family $E \to B$ equipped with the spin$^c$-structure $\mathfrak{s}$. Let $D \in K^0(B)$ be the families index of this spin$^c$-structure. Since 
\[
c_1(\mathfrak{s})|_X = (c + \pi^*(\eta))|_X = \epsilon_1 \xi_1 + \cdots + \epsilon_n \xi_n,
\]
we find (by the Atiyah-Singer index theorem) that the virtual rank of $D$ is given by
\[
d = \dfrac{ ( \epsilon_1 \xi_1 + \cdots + \epsilon_n \xi_n)^2 - n}{8} = \dfrac{n-n}{8} = 0.
\]
Therefore, Theorem \ref{thm:fsw1} says that $c_j(D) = 0 \in H^{2j}(B ; \mathbb{Q})$ for all $j > 0$. So $Ch(D) = 0$. Now by the families index theorem, we get that
\[
Ch(D) = 0 = \int_{E/B} e^{ c_1(\mathfrak{s})/2 } \hat{A}(T(E/B)) = 0.
\]
To finish, we observe that since $x_i|X = \xi_i = \zeta_i|X$, it follows that $\zeta_i = x_i + \pi^*( \eta_i)$ for some $\eta_i \in H^2( B ; \mathbb{Q})$. Therefore
\[
c_1(\mathfrak{s}) = \epsilon_1 x_1 + \epsilon_2 x_2 + \cdots \epsilon_n x_n + \pi^*( \eta + \epsilon_1\eta_1 + \cdots + \epsilon_n\eta_n)
\]
and hence
\begin{align*}
Ch(D) = 0 &= \int_{E/B} e^{( \epsilon_1 x_1 + \epsilon_2 x_2 + \cdots \epsilon_n x_n + \pi^*( \eta + \epsilon_1\eta_1 + \cdots + \epsilon_n\eta_n) )/2} \hat{A}(T(E/B)) \\
&= e^{\pi^*(\eta + \epsilon_1\eta_1 + \cdots + \epsilon_n\eta_n)/2} \int_{E/B} e^{( \epsilon_1 x_1 + \epsilon_2 x_2 + \cdots \epsilon_n x_n )/2} \hat{A}(T(E/B)).
\end{align*}
Multiplying through by $e^{-\pi^*(\eta + \epsilon_1\eta_1 + \cdots + \epsilon_n\eta_n)/2}$, we obtain the theorem.

\end{proof}

\begin{theorem}\label{thm:sw2}
Let $\pi : E \to B$ be a family with structure group $Diff_0(X)$. Fix $i \in \{1, \dots , n\}$ and for each $j \neq i$, $1 \le j \le n$, let $\epsilon_j \in \{1,-1\}$ be given. Set
\[
c = 3x_i + \sum_{j | j \neq i } \epsilon_j x_j.
\]
Then
\[
\int_{E/B} e^{c/2} \hat{A}(T(E/B)) = e^{u},
\]
where $u \in H^2(B ; \mathbb{Q})$ is given by
\[
u = \int_{E/B} \dfrac{1}{48}c^3 - \dfrac{1}{48} p_1 c.
\]
\end{theorem}
\begin{proof}
As in the proof of Theorem \ref{thm:sw}, it suffices to prove the result for families with the property that there exists classes $\zeta_1, \dots , \zeta_n \in H^2( E ; \mathbb{Z})$ such that $\zeta_i$ restricted to the fibres of $E$ equals $\xi_i$. Let
\[
c = 3x_i + \sum_{j | j \neq i } \epsilon_j x_j.
\]
Arguing as in the proof of Theorem \ref{thm:sw}, we have that there exists a spin$^c$ structure $\mathfrak{s}$ such that $c_1(\mathfrak{s}) = c + \pi^*(\eta)$ for some $\eta \in H^2(B ; \mathbb{Q})$. We apply Theorem \ref{thm:fsw1} to the family $E \to B$ equipped with the spin$^c$-structure $\mathfrak{s}$. Let $D \in K^0(B)$ be the families index of this spin$^c$-structure. Since 
\[
c_1(\mathfrak{s})|_X = 3\xi_1 + \sum_{j | j \neq i} \epsilon_j \xi_j,
\]
we find that the virtual rank of $D$ is given by
\[
d = \dfrac{ ( 3\xi_i + \sum_{j | j \neq i} \epsilon_j \xi_j)^2 - n}{8} = \dfrac{9 + (n-1)-n}{8} = \dfrac{8}{8} = 1.
\]
Therefore, \ref{thm:fsw1} says that $c_j(D) = 0 \in H^{2j}(B ; \mathbb{Q})$ for all $j > 1$. Using the Newton identities and the fact that $D$ has virtual rank $1$, we find that
\[
Ch(D) = e^{c_1(D)}.
\]
Now by the families index theorem, we get that
\[
Ch(D) = e^{c_1(D)} = \int_{E/B} e^{ c_1(\mathfrak{s})/2 } \hat{A}(T(E/B)) = \int_{E/B} e^{(c + \pi^*(\eta))/2} \hat{A}( T(E/B) ).
\]
Therefore
\begin{equation}\label{equ:int1}
\int_{E/B} e^{c/2} \hat{A}(T(E/B) ) = e^u,
\end{equation}
where $u = c_1(D) - \pi^*(\eta)/2$. Equating degree $2$ components in Equation (\ref{equ:int1}), we find that
\[
u = \int_{E/B} \dfrac{1}{48}c^3 - \dfrac{1}{48} p_1 c.
\]
\end{proof}

\section{Cohomology rings of families}\label{sec:cohomring}

In this section we apply Theorems \ref{thm:sw} \ref{thm:sw2} to obtain constrains on the structure of the cohomology ring $H^*(E ; \mathbb{Q})$ and on the Characteristic classes $p_1,e$. These constrains will then be used in Section \ref{sec:tautclasses} to deduce various properties of the tautological classes of $E$.

\begin{proposition}\label{prop:ident0}
Let $\pi : E \to B$ be a family with structure group $Diff_0(X)$. The following identities hold:

\begin{align}
0 &= \int_{E/B} x_i x_j x_i \; \; (\text{for distinct } i,j,k), \label{equ:sw1} \\
0 &= \int_{E/B} \left(x_i^3 + 3 \sum_{j | j \neq i} x_i x_j^2 - p_1 x_i \right) (\text{for each } i), \label{equ:sw2} \\
0 &= \int_{E/B} \left(x_i x_j^3 + x_i^3 x_j + 3 \sum_{k| k \neq i,j} x_i x_j x_k^2 - p_1 x_i x_j \right) \; \; (\text{for distinct } i,j),  \label{equ:sw3}\\
0 &= \int_{E/B} \left( \sum_i x_i^4 + 6 \sum_{i,j | i < j} x_i^2 x_j^2 + 3e^2 - 2p_1 \sum_i x_i^2\right). \label{equ:sw4}
\end{align}

\end{proposition}
\begin{proof}
Theorem \ref{thm:sw} gives
\[
\int_{E/B} e^{( \epsilon_1 x_1 + \cdots + \epsilon_n x_n)/2} \hat{A}(T(E/B)) = 0.
\]
Expanding out the exponential and integrating, we see that the degree $2m$ component of the left hand side has the form
\[
\sum_{ |I| \le m+2} \epsilon_I \alpha_{m,I}
\]
for some cohomology classes $\alpha_{m,I} \in H^{2m}(B ; \mathbb{Q})$, where the sum is over subsets of $\{1,2, \dots , n\}$ of size $\le m+2$ and
\[
\epsilon_{I} = \prod_{i \in I} \epsilon_{i}.
\]
Each $\epsilon_I$ can be thought of as a function $\epsilon_I : \{ 1 , -1 \}^n \to \mathbb{Q}$. Thought of this way, the $\{ \epsilon_I \}_{I}$ are linearly independent. Indeed they are the characters of $\mathbb{Z}_2^n,$ namely $\epsilon_I$ is the character of the $1$-dimensional representation in which the $i$-th generator of $\mathbb{Z}_2^n$ acts as $-1$ if $i \in I$ and as $+1$ if $i \notin I$. By linear independence, of the $\epsilon_I$, it follows that if
\[
\int_{E/B} e^{( \epsilon_1 x_1 + \cdots + \epsilon_n x_n)/2} \hat{A}(T(E/B)) = 0
\]
for all $\epsilon_1 , \dots , \epsilon_n \in \{1,-1\}$, then each class $\alpha_{m,I}$ must be zero. In degree $2m = 2$, we get
\begin{equation*}
0 = \int_{E/B} x_i x_j x_i \; \; (\text{for distinct } i,j,k),
\end{equation*}
which comes from $\alpha_{1, \{i,j,k\} } = 0$ and
\begin{equation*}
0 = \int_{E/B} \left(x_i^3 + 3 \sum_{j | j \neq i} x_i x_j^2 - p_1 x_i \right) (\text{for each } i),
\end{equation*}
which comes from $\alpha_{1, \{ i \} } = 0$. Notice that $\alpha_{1, \emptyset} = 0$ and $\alpha_{1 , \{ i,j\} } = 0$ hold automatically. In general it is clear that $\alpha_{m , I }=0$ holds automatically whenever $|I| \neq m \; ({\rm mod} \; 2)$. In degree $2m=4$, we get
\begin{align*}
0 &= \int_{E/B} \left(x_i x_j^3 + x_i^3 x_j + 3 \sum_{k| k \neq i,j} x_i x_j x_k^2 - p_1 x_i x_j \right) \; \; (\text{for distinct } i,j), \text{ and} \\
0 &= \int_{E/B} \left( \sum_i x_i^4 + 6 \sum_{i,j | i < j} x_i^2 x_j^2 + 3e^2 - 2p_1 \sum_i x_i^2\right), 
\end{align*}
which come from $\alpha_{2 , \{i,j\} } = 0$ and $\alpha_{2 , \emptyset } = 0$ respectively. 
\end{proof}
There is also an equation corresponding to $\alpha_{2, \{i,j,k,l\} } = 0$, but it turns out that this follows from (\ref{equ:sw1}), so doesn't give any further constraints. 

Recall that by Proposition \ref{prop:ident1}, the cohomology ring of $E$ is given by Equations (\ref{equ:21})-(\ref{equ:24}). We now use Proposition \ref{prop:ident0} to deduce further simplifications.
\begin{proposition}\label{prop:ring}
We have that
\[
D_{ij}^k = 0 \text{ whenever } k \neq i \text{ or } j.
\]
Henceforth we shall denote $D_{ij}^i$ by $D_{ij}$ (note that $D_{ij}^j = D_{ji}^j = D_{ji}$). Moreover, we have: 
\[
E_{ij} = -D_{ij}D_{ji} \text{ for all } i \neq j.
\]
Hence the cup product on $H^*(E ; \mathbb{Q})$ has the form
\begin{align}
x_i x_j &= D_{ij}x_i + D_{ji}x_j -D_{ij}D_{ji}, \label{equ:45} \\
x_i^2 & = \nu + \sum_{j | j\neq i} D_{ij} x_j + G_i, \label{equ:46} \\
x_i \nu &= G_i x_i - \sum_{j | j \neq i} D_{ij} D_{ji} x_j + J_i, \label{equ:47} \\
\nu^2 &= \sum_j J_j x_j + \omega. \label{equ:48}
\end{align}
\end{proposition}

\begin{proof}
Since $x_i x_j = \sum_{k} D_{ij}^k x_k + E_{ij}$, Equation (\ref{equ:sw1}) implies that
\[
D_{ij}^k = 0 \text{ whenever } k \neq i \text{ or } j.
\]
As stated in the proposition, we will henceforth denote $D_{ij}^i$ simply by $D_{ij}$. Next consider $p_1 \in H^4(E ; \mathbb{Q})$. Since $\int_{E/B} p_1 = 3n$, we may write $p_1$ in the form
\begin{equation}\label{equ:p10}
p_1 = 3n \nu + \sum_i \lambda_i x_i + \tau
\end{equation}
for some $\lambda_1, \dots \lambda_n \in H^2(B ; \mathbb{Q})$ and $\tau \in H^4(B ; \mathbb{Q})$. From (\ref{equ:sw2}) one finds
\begin{equation}\label{equ:lambda0}
\lambda_i = 3 \sum_{j | j \neq i} D_{ji}.
\end{equation}
Next, we note that (by Equation (\ref{equ:21}))
\[
\int_{E/B} x_i^3 x_j = \int_{E/B} x_i^2( D_{ij}x_i + D_{ji}x_j + E_{ij} ) = D_{ij}D_{ji} + E_{ij},
\]
\[
\int_{E/B} x_i x_j x_k^2 = \int_{E/B} (D_{ij}x_i + D_{ji}x_j + E_{ij})x_k^2 = D_{ij}D_{ki} + D_{ji}D_{kj} + E_{ij}
\]
for $i,j,k$ distinct, and
\begin{align*}
\int_{E/B} p_1 x_i x_j &= \int_{E/B} (3n\nu + \sum_{k} \lambda_k x_k + \tau )x_i x_j \text{ (by Equation (\ref{equ:p10}))} \\
&= 3nE_{ij} + D_{ij}\lambda_i + D_{ji}\lambda_j \text{ (by Equation (\ref{equ:21}))} \\
&= 3nE_{ij} + 6D_{ij}D_{ji} + 3\sum_{k|k \neq i,j} (D_{ij}D_{ki} + D_{ji}D_{kj}) \text{ (by Equation (\ref{equ:lambda0}))}
\end{align*}
Hence (\ref{equ:sw3}) gives:
\[
2( D_{ij}D_{ji} + E_{ij}) + 3\sum_{k|k \neq i,j} ( D_{ij}D_{ki} + D_{ji}D_{kj} + E_{ij}) - 3nE_{ij} - 6D_{ij}D_{ji} - 3\sum_{k|k \neq i,j} (D_{ij}D_{ki} + D_{ji}D_{kj}) = 0,
\]
which simplifies to
\[
2( D_{ij}D_{ji} + E_{ij}) + 3(n-2)E_{ij} - 3nE_{ij} - 6D_{ij}D_{ji}  = 0,
\]
or
\[
-4( D_{ij}D_{ji} + E_{ij}) = 0.
\]
So we have
\[
E_{ij} = -D_{ij}D_{ji}.
\]
\end{proof}

A surprising consequence of Proposition \ref{prop:ring} is that the equation for $x_i x_j$,
\[
x_i x_j = D_{ij}x_i + D_{ji}x_j - D_{ij}D_{ji},
\]
can be written more compactly as
\[
(x_i - D_{ji})(x_j - D_{ij}) = 0.
\]

From Proposition \ref{prop:ring}, the cup product on $H^*(E ; \mathbb{Q})$ is determined by classes $D_{ij}, G_i, J_i, \omega$. However there are certain constraints that these classes must satisfy arising from associativity of the cup product.

\begin{proposition}
The classes $D_{ij}$, $G_i$, $J_i$ and $\omega$ satisfy the following equations:
\begin{align}
(D_{ij}-D_{kj})(D_{ik}-D_{jk}) &= 0 \text{ (for distinct } i,j,k), \label{equ:assoc1} \\
G_i + G_j + \sum_{k| k \neq i,j} D_{ik}D_{jk} &= D_{ij}^2 + D_{ji}^2 \text{ (for distinct } i,j), \label{equ:assoc2} \\
J_j + D_{ij}G_j - \sum_{k | k \neq i,j} D_{ik} D_{jk}D_{kj} &= D_{ij}G_i -D_{ij} D_{ji}^2 \text{ (for distinct } i,j), \label{equ:assoc3} \\
\sum_{j| j \neq i} J_j D_{ij} + \omega &= G_i^2 + \sum_{j| j \neq i} D_{ij}^2 D_{ji}^2 \text{ (for all } i). \label{equ:assoc4}
\end{align}
\end{proposition}
\begin{proof}
Recall the cup product on $H^*(E ; \mathbb{Q})$ is given in Proposition \ref{prop:ring}. Associativity of this product gives contraints on $D_{ij}, G_i, J_i, \omega$. Let $i,j,k$ be distinct. From
\[
(x_i x_j) x_k = x_i (x_j x_k)
\]
we obtain Equation (\ref{equ:assoc1}). From
\[
(x_i^2)x_j = x_i ( x_i x_j)
\]
we obtain Equation (\ref{equ:assoc2}). From
\[
(x_i^2) \nu = x_i( x_i \nu)
\]
we obtain Equation (\ref{equ:assoc3}) and from
\[
x_i (\nu^2) = (x_i \nu) \nu
\]
we obtain Equation (\ref{equ:assoc4}). 
\end{proof}

Conversely, it can be checked that Equations (\ref{equ:assoc1}) - (\ref{equ:assoc4}) imply associativity of the product given in Proposition \ref{prop:ring}, so there are no further equations that can be obtained from associativity alone.

\begin{proposition}\label{prop:xcubed}
For each $i=1, \dots , n$, we have
\[
x_i^3 = B_i x_i + C_i,
\]
where
\begin{align}
B_i &= 2G_i + \sum_{j | j \neq i } D_{ij}^2, \label{equ:bg} \\
C_i &= J_i - \sum_{j | j \neq i} D_{ij}^2 D_{ji}. \label{equ:cj}
\end{align}
\end{proposition}
\begin{proof}
We have
\[
x_i^2 = \nu + \sum_{j | j \neq i } D_{ij} x_j + G_i.
\]
Multiplying both sides by $x_i$ and using $x_i \nu = G_i x_i - \sum_{j | j \neq i} D_{ij}D_{ji} x_j + J_i$, we get
\begin{align*}
x_i^3 &= x_i \nu + \sum_{j | j \neq i} D_{ij}(x_ix_j) + G_i x_i \\
&= 2G_i x_i - \sum_{j | j \neq i } D_{ij} D_{ji} x_j + J_i + \sum_{j | j \neq i} D_{ij}( D_{ij}x_i + D_{ji} x_j - D_{ij}D_{ji} ) \\
&= (2G_i + \sum_{j | j \neq i} D_{ij}^2 )x_i + J_i - \sum_{j | j \neq i} D_{ij}^2 D_{ji} \\
&= B_i x_i + C_i. 
\end{align*}

\end{proof}

\begin{proposition}\label{prop:dcubic}
For distinct $i,j$, we have
\[
D_{ji}^3 = B_i D_{ji} + C_i.
\]
In other words, $D_{ji}$ satisfies the same cubic equation as $x_i$.
\end{proposition}
\begin{proof}
Recall that $x_i x_j = D_{ij}x_i + D_{ji} x_j - D_{ij}D_{ji}$. Therefore
\[
x_i ( x_j - D_{ij}) = D_{ji} ( x_j - D_{ij}).
\]
By repeated application of this identity we see that $x_i^k ( x_j - D_{ij}) = D_{ji}^k(x_j - D_{ij})$ for any $k \ge 0$. Using this and Proposition \ref{prop:xcubed}, we have
\[
D_{ji}^3( x_j - D_{ij}) = x_{ij}^3(x_j - D_{ij}) = ( B_i x_i + C_i) (x_j-D_{ij}) = (B_i D_{ji} + C_i) ( x_j-D_{ij}).
\]
Multiplying both sides by $x_j$ and integrating over the fibres gives
\[
D_{ji}^3 = B_i D_{ji} + C_i.
\]
\end{proof}

In order to compute tautological classes $E$, we need to determine $p_1, e$ as elements of $H^4(E ; \mathbb{Q})$. This is carried out in the next few propositions. First we consider $p_1$.

\begin{proposition}\label{prop:p1}
There exists a class $\mu \in H^4(B ; \mathbb{Q})$ such that
\[
p_1 = 3(x_1^2 + \cdots + x_n^2) + \mu.
\]
\end{proposition}
\begin{proof}
By the signature theorem, $\int_{E/B} p_1 = 3n$. Therefore
\[
p_1 = 3(x_1^2 + \cdots + x_n^2) + \sum_{i=1}^n d_i x_i + \mu
\]
for some $d_i \in H^2(B ; \mathbb{Q})$ and some $\mu \in H^4(B ; \mathbb{Q})$. Then using Equation (\ref{equ:sw2}), we find that $d_i = 0$ for all $i$ and hence
\[
p_1 = 3(x_1^2 +  \cdots + x_n^2) + \mu.
\]
\end{proof}

\begin{proposition}\label{prop:bk}
For each $i$, we have
\[
-2B_i + 3 \sum_{j | j \neq i } D_{ij}^2 + \mu = 0.
\]
\end{proposition}
\begin{proof}
We use Theorem \ref{thm:sw2}. For $j \neq i$, let $\epsilon_j \in \{ 1 , -1\}$ and set
\[
c = 3x_i + \sum_{j | j \neq i } \epsilon_j x_j.
\]
Then
\begin{equation}\label{equ:eu}
\int_{E/B} e^{c/2} \hat{A}(T(E/B)) = e^{u},
\end{equation}
where 
\[
u = \int_{E/B} \dfrac{1}{48}c^3 - \dfrac{1}{48} p_1 c.
\]
Since $u$ is cubic in $c$, it can be expanded in the form
\[
u = u_0 + \sum_{j | j \neq i} \epsilon_j u_j + \sum_{j,k | i,j,k \text{ distinct} } \epsilon_j \epsilon_k u_{jk} + \sum_{j,k,l | i,j,k,l \text{ distinct} } \epsilon_j \epsilon_k \epsilon_l u_{jkl}
\]
for some $u_0, u_{j}, u_{jk}, u_{jkl} \in H^2( B ; \mathbb{Q})$. From Proposition \ref{prop:p1} we find that
\[
u = \dfrac{1}{48} \int_{E/B} \left( (3x_i + \sum_{j | j \neq i } \epsilon_j x_j)^3 - (3x_i + \sum_{j | j \neq i } \epsilon_j x_j)( 3(x_1^2 + \cdots + x_n^2) + \mu ) \right).
\]
Expanding this out, we find
\[
u_0 = \dfrac{1}{48} \int_{E/B} \left( 27 x_i^3 + 9x_i \sum_{j | j \neq i} x_j^2 - 9x_i^3 - 9x_i \sum_{j | j \neq i } x_j^2 \right) = \dfrac{3}{8} \int_{E/B} x_i^3 = 0,
\]
\begin{align*}
u_j &= \dfrac{1}{48} \int_{E/B} \left( 27 x_i^2 x_j + x_j^3 + 3x_j \sum_{k | k \neq i,j} x_k^2 -3x_j^3 - 3x_j \sum_{k | k \neq j } x_k^2 \right) \\
&= \dfrac{1}{48} \int_{E/B} \left( 27 x_i^2 x_j + 3x_j \sum_{k | k \neq i,j} x_k^2 - 3x_j \sum_{k | k \neq i, j } x_k^2  - 3x_j x_i^2 \right) \\
&= \dfrac{1}{48} \int_{E/B} \left( 24 x_i^2 x_j \right) \\
&= \dfrac{1}{2} \int_{E/B} x_i^2 x_j = \dfrac{1}{2} D_{ij}.
\end{align*}
We also have
\[
u_{ij} = \dfrac{1}{48} \int_{E/B} \left( 18 x_i x_j x_k \right ) = 0,
\]
and
\[
u_{ijk} = \dfrac{1}{48} \int_{E/B} \left( 6 x_i x_j x_k \right) = 0.
\]
So
\[
u = \dfrac{1}{2} \sum_{j | j \neq i } \epsilon_j D_{ij}.
\]
It follows that
\begin{equation}\label{equ:u2}
\dfrac{1}{2}u^2 = \dfrac{1}{8}\left( \sum_{j | j \neq i} D_{ij}^2 \; \; + \; \; 2 \! \! \! \! \! \! \! \! \! \! \! \! \sum_{j,k | i,j,k \text{ distinct} }  \! \! \! \! \! \! \epsilon_j \epsilon_k D_{ij}D_{ik} \right).
\end{equation}
Equating degree $4$ components of Equation (\ref{equ:eu}), we get that
\begin{equation}\label{equ:usquared}
\dfrac{1}{2}u^2 = \int_{E/B} \left( \dfrac{1}{384} c^4 - \dfrac{1}{192} c^2 p_1 + \dfrac{1}{5760}(7p_1^2-4e^2) \right).
\end{equation}
Observe that
\[
c - 2x_i = x_i + \sum_{j | j \neq i } \epsilon_j x_j.
\]
Therefore, Theorem \ref{thm:sw} gives:
\[
\int_{E/B} e^{(c-2x_i)/2} \hat{A}( T(E/B) ) = 0.
\]
Extracting the degree $4$ terms, we find
\begin{equation}\label{equ:degree4}
\int_{E/B} \left( \dfrac{1}{384} (c-2x_i)^4 - \dfrac{1}{192} (c-2x_i)^2 p_1 + \dfrac{1}{5760}(7p_1^2-4e^2) \right) = 0.
\end{equation}
Combining Equations (\ref{equ:usquared}) and (\ref{equ:degree4}) gives:
\[
\dfrac{1}{2}u^2 = \int_{E/B} \left( \dfrac{1}{384} [c^4 - (c-2x_i)^4] - \dfrac{1}{192} [c^2-(c-2x_i)^2] p_1 \right).
\]
The right hand side can be expanded out in the form $\sum_{I} \epsilon_I v_I$ for some $v_I \in H^4(B ; \mathbb{Q})$. We will be interested in the constant term $v_{\emptyset}$ (here ``constant" means independent of the $\epsilon_j$). Comparing with Equation (\ref{equ:u2}), we have
\[
v_{\emptyset} = \dfrac{1}{8} \sum_{j | j \neq i} D_{ij}^2.
\]
Using Proposition \ref{prop:p1}, we find
\[
-\dfrac{1}{192} \int_{E/B} [ c^2 - (c-2x_i)^2] p_1 = -\dfrac{1}{48} \int_{E/B} \left( (3x_i+\sum_{j | j \neq i} \epsilon_j x_j )x_i - x_i^2 \right)\left( 3(x_1^2 + \cdots + x_n^2) + \mu \right).
\]
The constant term in this expression is
\begin{equation}\label{equ:const1}
\begin{aligned}
-\dfrac{1}{48} \int_{E/B} ( 3x_i^2 - x_i^2)( 3(x_1^2 + \cdots + x_n^2) + \mu ) &= -\dfrac{1}{24} \int_{E/B} ( 3x_i^4 + 3x_i^2 \sum_{j | j \neq i} x_j^2 + \mu ) \\
&= -\dfrac{1}{8}B_i  - \dfrac{1}{8} \sum_{j | j \neq i} \int_{E/B} x_i^2 x_j^2 - \dfrac{1}{24} \mu.
\end{aligned}
\end{equation}
Also, we have
\[
\dfrac{1}{384} \int_{E/B} ( c^4 - (c-2x_i)^4 ) = \dfrac{1}{48}\int_{E/B} \left( c^3 x_i - 3c^2 x_i^2 + 4c x_i^3 - 2x_i^4 \right).
\]
The constant term in this expression is
\begin{equation}\label{equ:const2}
\begin{aligned}
& \dfrac{1}{48} \left( -2B_i + 12 B_i - 27B_i - 3 \sum_{j | j \neq i} \int_{E/B} x_i^2 x_j^2 + 27 B_i + 9 \sum_{j | j \neq i} \int_{E/B} x_i^2 x_j^2 \right) \\
& \; \; \; \; = \dfrac{1}{48}\left( 10B_i + 6\sum_{j | j \neq i} \int_{E/B} x_i^2 x_j^2 \right) \\
& \; \; \; \; = \dfrac{5}{24}B_i + \dfrac{1}{8}\sum_{j | j \neq i} \int_{E/B} x_i^2 x_j^2.
\end{aligned}
\end{equation}
Combining the constant terms from (\ref{equ:const1}) and (\ref{equ:const2}) and equating this to $v_{\emptyset}$, we obtain
\begin{align*}
\dfrac{1}{8} \sum_{j | j \neq i} D_{ij}^2 &= -\dfrac{1}{8}B_i  - \dfrac{1}{8} \sum_{j | j \neq i} \int_{E/B} x_i^2 x_j^2 - \dfrac{1}{24} \mu + \dfrac{5}{24}B_i + \dfrac{1}{8}\sum_{j | j \neq i} \int_{E/B} x_i^2 x_j^2 \\
&= \dfrac{1}{12} B_i - \dfrac{1}{24}\mu.
\end{align*}
Hence
\[
3 \sum_{j | j \neq i } D_{ij}^2 = 2B_i - \mu.
\]
\end{proof}

We now consider the Euler class $e$. To determine $e$ we will make use of the following result \cite[Lemma 2.3]{rw}:

\begin{lemma}\label{lem:trace}
For any $\alpha \in H^{ev}(E ; \mathbb{Q})$, we have
\[
\int_{E/B} \alpha e = Trace \left( \alpha : H^{ev}(E ; \mathbb{Q}) \to H^{ev}(E ; \mathbb{Q}) \right)
\]
where we view $H^{ev}(E ; \mathbb{Q})$ as a finite dimensional free module over $H^{ev}(B ; \mathbb{Q})$.
\end{lemma}

\begin{proposition}\label{prop:euler}
We have
\[
e = 2\nu + x_1^2 + \cdots + x_n^2.
\]
\end{proposition}
\begin{proof}
From Lemma \ref{lem:trace} and we have that
\[
\int_{E/B} e x_i = Trace \left( x_i : H^{ev}(E ; \mathbb{Q}) \to H^{ev}(E ; \mathbb{Q}) \right).
\]
We compute the trace of $x_i$ using the basis $1 , x_1, \dots , x_n , \nu$. From Equation (\ref{equ:45}) the coefficient of $x_j$ in $x_ix_j$ for $i \neq j$ is $D_{ji}$. From Equation (\ref{equ:46}) the coefficient of $x_i$ in $x_i^2$ is zero and from Equation (\ref{equ:47}) the coefficient of $\nu$ in $x_i \nu$ is zero. Hence
\[
\int_{E/B} e x_i = \sum_{j | j \neq i } D_{ji}.
\]
Similarly we compute the trace of $\nu$ using the basis $1,x_1 , \dots , x_n , \nu$. From Equation (\ref{equ:47}) the coefficient of $x_i$ in $\nu x_i$ is $G_i$ and from Equation (\ref{equ:48}) the coefficient of $\nu$ in $\nu^2$ is zero. Hence
\[
\int_{E/B} e \nu = \sum_{i=1}^n G_i.
\]
From $\int_{E/B} e = \chi(X) = 2+n$, it follows that
\[
e = 2\nu + x_1^2 + \cdots + x_n + \sum_{i=1}^n v_i x_i + w
\]
for some $v_i \in H^2(B ; \mathbb{Q})$ and some $w \in H^4( B ; \mathbb{Q})$. By direct computation one finds that
\[
\int_{E/B} e x_i = v_i + \sum_{j | j \neq i} D_{ji}
\]
and
\[
\int_{E/B} e \nu = \sum_{i=1}^n G_i + w.
\]
Therefore $v_i = 0$ for all $i$ and $w=0$, so that
\[
e = 2 \nu + x_1^2 + \cdots + x_n^2.
\]
\end{proof}

\section{Tautological classes}\label{sec:tautclasses}

In this section we use the results from Section \ref{sec:cohomring} on the structure of the cohomology ring of $E$ in order to compute tautological classes. The computation involves many complicated fibre integrals and hence we find it useful to first prove a general result about such integrals. The following result may be thought of as a kind of diagonalisation theorem for fibre integrals.

\begin{proposition}[Integration formula]\label{prop:intf}
Let $f( t_1 , \dots , t_n)$ be a polynomial in $t_1, \dots , t_n$ with coefficients in $H^{ev}(B ; \mathbb{Q})$. Then
\begin{equation}\label{equ:intformula}
\int_{E/B} f(x_1 , \dots , x_n) = \sum_{j=1}^n \int_{E/B} f_j(x_j),
\end{equation}
where
\[
f_j(x_j) = f( D_{j1} , \dots , D_{j (j-1) } , x_j , D_{j (j+1)} , \dots , D_{jn} ).
\]

\end{proposition}
\begin{proof}
We prove this by induction. Let $m \le n$. Suppose we have shown that (\ref{equ:intformula}) holds whenever $f$ is a polynomial in $m$ or fewer of the variables $t_1, \dots ,t_m$. The base case $m=1$ is trivially true. Suppose the result holds for $1 \le m < n$. To prove the $m+1$ case in general, it is enough to prove it for polynomials of the form $t_j^a f( t_{i_1} , \dots , t_{i_m} )$ for some $j , i_1, \dots , i_m$ and some $a \ge 1$. Without loss of generality, it is enough to prove it for polynomials of the form $t_{m+1}^a f(t_1 , \dots , t_m)$, since we can get the result for $t_j^a f( t_{i_1} , \dots , t_{i_m} )$ from this by reordering the indices. Note that from $x_{m+1} x_1 = D_{(m+1) 1}x_{m+1} D_{1 (m+1)} x_1 - D_{(m+1) 1}D_{1 (m+1)}$, we get that
\[
x_{m+1}( x_1 - D_{(m+1) 1} ) = D_{1 (m+1)} (x_1 - D_{(m+1) 1}).
\]
Next, by the division algorithm we can write
\[
f(x_1 , \dots , x_m ) = f( D_{(m+1) 1} , x_2 , \dots , x_m) + (x_1 - D_{(m+1) 1}) g( x_1 , \dots , x_m)
\]
for some polynomial $g(x_1 , \dots , x_m)$. Then
\begin{align*}
& \int_{E/B} x_{m+1}^a f(x_1 , \dots , x_m ) = \int_{E/B} x_{m+1}^a \left( f( D_{(m+1) 1} , x_2 , \dots , x_m) + (x_1 - D_{(m+1) 1}) g( x_1 , \dots , x_m) \right) \\
& \quad \quad \quad = \int_{E/B} x_{m+1}^a  f( D_{(m+1) 1} , x_2 , \dots , x_m ) ) + D_{1 (m+1)}^a (x_1-D_{(m+1) 1}) g(x_1, \dots , x_m).
\end{align*}
Both terms involve at most $m$ variables, so by induction we get
\begin{align*}
& \int_{E/B} x_{m+1}^a f(D_{(m+1) 1} , x_2 , \dots , x_m) = \int_{E/B} x_{m+1}^a f( D_{(m+1) 1} , D_{(m+1) 2} , \dots , D_{(m+1) m} ) \\
& \quad \quad \quad + \int_{E/B} \sum_{j=1}^m D_{j (m+1)}^a f( D_{(m+1) 1} , D_{j2} , \dots , D_{j (j-1)} , x_j , D_{j (j+1)} , \dots , D_{j m})
\end{align*}
and
\begin{align*}
& \int_{E/B} D_{1 (m+1)}^a (x_1-D_{(m+1) 1}) g(x_1, \dots , x_m) \\
& \; \; \; \; = \int_{E/B} \sum_{j=1}^m D_{1 (m+1)}^a (D_{j1} - D_{(m+1) 1}) g( D_{j1} , \dots, D_{j (j-1)} , x_j , D_{j (j+1) } , \dots , D_{j m} ) \\
& \; \; \; \; = \int_{E/B} \sum_{j=1}^m D_{j (m+1)}^a (D_{j1} - D_{(m+1) 1}) g( D_{j1} , \dots, D_{j (j-1)} , x_j , D_{j (j+1) } , \dots , D_{j m} )
\end{align*}
To get the last line we used the following special case of Equation (\ref{equ:assoc1})
\[
(D_{1 (m+1)} - D_{j (m+1)})( D_{j1} - D_{(m+1) 1} ) = 0
\]
to deduce that
\[
D_{1 (m+1)}( D_{j1} - D_{(m+1) 1} ) = D_{j (m+1)} ( D_{j1} - D_{(m+1) 1})
\]
and hence by repeated application of this identity, that
\[
D_{1 (m+1)}^a( D_{j1} - D_{(m+1) 1} ) = D_{j (m+1)}^a ( D_{j1} - D_{(m+1) 1}).
\]
Adding these two equalities we get
\begin{align*}
\int_{E/B} x_{m+1}^a f(x_1 , \dots , x_m ) &= \int_{E/B} x_{m+1}^a f( D_{(m+1) 1} , D_{(m+1) 2} , \dots , D_{(m+1) m} ) \\
& + \int_{E/B} \sum_{j=1}^m D_{j (m+1)}^a f( D_{(m+1) 1} , D_{j2} , \dots , D_{j (j-1)} , x_j , D_{j (j+1)} , \dots , D_{j m}) \\ 
& + \int_{E/B} \sum_{j=1}^m D_{j m+1}^a (D_{j1} - D_{(m+1) 1}) g( D_{j1} , \dots, D_{j (j-1)} , x_j , D_{j (j+1) } , \dots , D_{j m}) \\
&= \int_{E/B} x_{m+1}^a f( D_{(m+1) 1} , D_{(m+1) 2} , \dots , D_{(m+1) m}) \\
& + \int_{E/B} \sum_{j=1}^m D_{j (m+1)}^a f( D_{j1} , \dots , D_{j (j-1)} , x_j , D_{j (j+1)} , \dots , D_{jm} ).  
\end{align*}
This proves the inductive step and so we are done.
\end{proof}

The computation of the tautological classes $\kappa_{p_1^a e^b}$ is best handled by considering separately the cases where $b$ is even or odd. Because of this, we wish to have an expression for $e^2$ in terms of $x_1, \dots, x_n$ in which there are no cross-terms $x_i x_j$, $i \neq j$. The following proposition gives such an expression.

\begin{proposition}\label{prop:e2}
We have
\[
e^2 = \sum_{i=1}^n (3x_i^2 - B_i)^2 + \lambda
\]
where $\lambda$ satisfies
\[
\lambda = -\sum_{j | j \neq i} (3D_{ij}^2 - B_j)^2
\]
for all $i$.
\end{proposition}
\begin{proof}
Throughout this proof we will make repeated use of Equations (\ref{equ:45}), (\ref{equ:46}), (\ref{equ:47}), (\ref{equ:48}) to compute cup products. We will also use Proposition \ref{prop:xcubed} to express third and higher powers of $x_i$ in terms of $x_i$ and $x_i^2$. In particular, we have that 
\[
x_i^4 = B_i x_i^2 + C_i x_i.
\]
We will also use Equations (\ref{equ:bg}) and (\ref{equ:cj}) to re-write any instances of $G_i$ or $J_i$ that appear in the calculation in terms of $B_i,C_i, D_{ij}, D_{ji}$.

From Proposition \ref{prop:euler} we have 
\[
e = 2\nu + x_1^2 + \cdots + x_n^2.
\]
Squaring both sides and simplifying, we find
\begin{align*}
e^2 &= 3 \sum_i B_i x_i^2 + 9 \sum_i C_i x_i + 3 \sum_{i,j | i \neq j } D_{ij}^2 D_{ji}^2 + 4 \omega \\
&= \sum_{i=1}^n (3x_i^2 - B_i)^2 + \lambda
\end{align*}
where
\[
\lambda = 3 \sum_{i,j | i \neq j} D_{ij}^2 D_{ji}^2 + 4 \omega - \sum_{i} B_i^2.
\]
It remains to show that $\lambda = -\sum_{j | j \neq i} (3D_{ij}^2 - B_j)^2$. To show this, we will calculate $\int_{E/B} e^2 x_i^2$ in two different ways. First, a direct computation gives
\begin{align*}
\int_{E/B} e^2 x_i^2 &= \int_{E/B} (2\nu + x_1^2 + \cdots + x_n^2)^2 x_i^2 \\
&= \int_{E/B} 4 \nu^2 x_i^2 + 4 \int_{E/B} \nu (x_1^2 + \cdots + x_n^2)x_i^2 + \int_{E/B} (x_1^2 + \cdots + x_n^2)^2 x_i^2.
\end{align*}
Now
\begin{align*}
\int_{E/B} \nu^2 x_i^2 &= \int_{E/B} (\nu x_i)^2 = \int_{E/B} (G_i x_i - \sum_{j | j \neq i} D_{ij} D_{ji} x_j + J_i)^2 \\
&= G_i^2 + \sum_{j | j \neq i} D_{ij}^2 D_{ji}^2
\end{align*}
and
\begin{align*}
\int_{E/B} \nu (x_1^2 + \cdots + x_n^2) x_i^2 &= \int_{E/B} \nu x_i^4 + \sum_{j | j \neq i} \int_{E/B} \nu x_j^2 x_i^2 \\
&= \int_{E/B} \nu( B_i x_i^2 + C_i x_i) + \sum_{j | j \neq i} \int_{E/B} \nu ( D_{ij} x_i + D_{ji} x_j - D_{ij} D_{ji})^2 \\
&= B_i G_i + \sum_{j | j \neq i} ( D_{ij}^2 G_i + D_{ji}^2 G_j - D_{ij}^2 D_{ji}^2 ).
\end{align*}
Also, from Proposition \ref{prop:intf}, we have
\begin{align*}
\int_{E/B} (x_1^2 + \cdots + x_n^2)^2 x_i^2 &= \int_{E/B} x_i^2( x_i^2 + \sum_{j | j \neq i} D_{ij}^2 )^2 + \sum_{j | j \neq i} D_{ji}^2 \int_{E/B} ( x_j^2 + \sum_{k | k \neq j} D_{jk}^2 )^2 \\
&= B_i^2 + 2 B_i \sum_{j | j \neq i} D_{ij}^2 + ( \sum_{j | j \neq i} D_{ij}^2)^2 + \sum_{j | j \neq i} D_{ji}^2 ( B_j + 2 \sum_{k | k \neq j} D_{jk}^2 ),
\end{align*}
where we used $\int_{E/B} x_i^4 = B_i$, $\int_{E/B} x_i^6 = B_i^2$ to get the last line. Putting these together, we find

\begin{align*}
& \int_{E/B} e^2 x_i^2 \\
&= \int_{E/B} 4 \nu^2 x_i^2 + 4 \int_{E/B} \nu (x_1^2 + \cdots + x_n^2)x_i^2 + \int_{E/B} (x_1^2 + \cdots + x_n^2)^2 x_i^2 \\
&= 4G_i^2 + 4\sum_{j | j \neq i} D_{ij}^2 D_{ji}^2 + 4B_i G_i + \sum_{j | j \neq i} ( 4D_{ij}^2 G_i + 4D_{ji}^2 G_j - 4D_{ij}^2 D_{ji}^2 ) \\
& \; \; \; \; + B_i^2 + 2 B_i \sum_{j | j \neq i} D_{ij}^2 + ( \sum_{j | j \neq i} D_{ij}^2)^2 + \sum_{j | j \neq i} D_{ji}^2 B_j + 2 \! \! \! \! \! \! \sum_{j , k | i,j,k \text{ distinct}}  \! \! \! \! \! \! D_{ji}^2 D_{jk}^2 + 2 \sum_{j | j \neq i } D_{ji}^4  \\
&= \left( B_i - \sum_{j | j \neq i } D_{ij}^2 \right)^2 + 4\sum_{j | j \neq i} D_{ij}^2 D_{ji}^2 + 2B_i^2 - 2B_i \sum_{j | j \neq i} D_{ij}^2 \\ 
& \; \; \; \; + 2 \sum_{j | j \neq i} D_{ij}^2 B_i  -2\sum_{j | j \neq i} \left( \sum_{k | k \neq i} D_{ij}^2 D_{ik}^2 \right) + 2\sum_{j | j \neq i}D_{ji}^2 B_j - 2\sum_{j | j \neq i}\left( \sum_{k | k \neq j} D_{ji}^2 D_{jk}^2 \right)  - 4 \sum_{j | j \neq i} D_{ij}^2 D_{ji}^2  \\
& \; \; \; \; + B_i^2 + 2 B_i \sum_{j | j \neq i} D_{ij}^2 + ( \sum_{j | j \neq i} D_{ij}^2)^2 + \sum_{j | j \neq i} D_{ji}^2 B_j + 2 \! \! \! \! \! \! \sum_{j , k | i,j,k \text{ distinct}}  \! \! \! \! \! \! D_{ji}^2 D_{jk}^2 + 2 \sum_{j | j \neq i } D_{ji}^4 
\end{align*}
\begin{align*}
&= \left( B_i - \sum_{j | j \neq i } D_{ij}^2 \right)^2 + 4\sum_{j | j \neq i} D_{ij}^2 D_{ji}^2 + 3B_i^2 \\ 
& \; \; \; \;  -2\sum_{j | j \neq i} \left( \sum_{k | k \neq i} D_{ij}^2 D_{ik}^2 \right)  + 3\sum_{j | j \neq i}D_{ji}^2 B_j - 2\sum_{j | j \neq i}\left( \sum_{k | k \neq j} D_{ji}^2 D_{jk}^2 \right)  - 4 \sum_{j | j \neq i} D_{ij}^2 D_{ji}^2  \\
& \; \; \; \; + 2 B_i \sum_{j | j \neq i} D_{ij}^2 + ( \sum_{j | j \neq i} D_{ij}^2)^2 + 2 \! \! \! \! \! \! \sum_{j , k | i,j,k \text{ distinct}}  \! \! \! \! \! \! D_{ji}^2 D_{jk}^2 + 2 \sum_{j | j \neq i } D_{ji}^4 \\
&= 4B_i^2 + 3\sum_{j | j \neq i}D_{ji}^2 B_j + 2(\sum_{j | j \neq i } D_{ij}^2)^2 + 4\sum_{j | j \neq i} D_{ij}^2 D_{ji}^2 \\ 
& \; \; \; \;  -2\sum_{j | j \neq i} \left( \sum_{k | k \neq i} D_{ij}^2 D_{ik}^2 \right)  - 2\sum_{j | j \neq i} \left( \sum_{k | k \neq j} D_{ji}^2 D_{jk}^2 \right) - 4 \sum_{j | j \neq i} D_{ij}^2 D_{ji}^2 + 2 \! \! \! \! \! \! \sum_{j , k | i,j,k \text{ distinct}}  \! \! \! \! \! \! D_{ji}^2 D_{jk}^2 + 2 \sum_{j | j \neq i } D_{ji}^4 \\
&= 4B_i^2 + 3\sum_{j | j \neq i}D_{ji}^2 B_j + 2 \! \! \! \! \! \!  \sum_{j,k | i,j,k \text{ distinct} } \! \! \! \! \! \!  D_{ij}^2 D_{ik}^2 + 2 \sum_{ j | j \neq i } D_{ij}^4 \\
& \; \; \; \; -2 \! \! \! \! \! \! \sum_{j,k | i,j,k \text{ distinct} } \! \! \! \! \! \! D_{ij}^2 D_{ik}^2 - 2 \sum_{j | j \neq i } D_{ij}^4 -2 \! \! \! \! \! \! \sum_{j,k | i,j,k \text{ distinct} } \! \! \! \! \! \! D_{ji}^2 D_{jk}^2 - 2 \sum_{j | j \neq i } D_{ji}^4 + 2 \! \! \! \! \! \! \sum_{j , k | i,j,k \text{ distinct}}  \! \! \! \! \! \! D_{ji}^2 D_{jk}^2 + 2 \sum_{j | j \neq i } D_{ji}^4  \\
&= 4B_i^2 + 3\sum_{j | j \neq i}D_{ji}^2 B_j.
\end{align*}

So we have shown
\begin{equation}\label{equ:eexx}
\int_{E/B} e^2 x_i^2 = 4B_i^2 + 3 \sum_{j | j \neq i} D_{ji}^2 B_j.
\end{equation}

Next, take the equality
\[
e^2 = \sum_{j=1}^n (3x_j^2 - B_j)^2 + \lambda = 3 \sum_j B_j x_j^2 + 9 \sum_j C_j x_j + \sum_j B_j^2 + \lambda,
\]
multiply both sides by $x_i^2$ and integrate to obtain
\begin{align*}
\int_{E/B} e^2 x_i^2 &= 3\int_{E/B} B_i x_i^4 + 3 \sum_{j | j \neq i} \int_{E/B} B_j x_i^2 x_j^2 + 9 \sum_{j | j \neq i} C_j D_{ij} + \sum_{j | j \neq i} B_j^2 + B_i^2 + \lambda \\
&= 4 B_i^2 + 3 \sum_{j | j \neq i} B_j \int_{E/B} x_i^2 x_j^2 + 9 \sum_{j | j \neq i} C_j D_{ij} + \sum_{j | j \neq i} B_j^2 + \lambda.
\end{align*}
From Proposition \ref{prop:intf}, we find
\[
\int_{E/B} x_i^2 x_j^2 = D_{ij}^2 + D_{ji}^2,
\]
so
\begin{align*}
\int_{E/B} e^2 x_i^2 &= 4 B_i^2 + \sum_{j | j \neq i} \left( 3B_jD_{ij}^2 + 3B_j D_{ji}^2 +9 C_j D_{ij} + B_j^2\right) + \lambda \\
&= 4 B_i^2 + \sum_{j | j \neq i} \left( 3B_j D_{ij}^2 + 3B_j D_{ji}^2 + 9 D_{ij}^4 - 9D_{ij}^2 B_j + B_j^2 \right) + \lambda \\
&= 4 B_i^2 + \sum_{j | j \neq i} \left( 9D_{ij}^4 - 6 D_{ij}^2 B_j + B_j^2 + 3B_j D_{ji}^2 \right) + \lambda \\
&= 4 B_i^2 + 3\sum_{j | j \neq i}  B_j D_{ji}^2+ \left( \sum_{j | j \neq i} ( 3D_{ij}^2 - B_j)^2 + \lambda \right),
\end{align*}
where we made use of Proposition \ref{prop:dcubic} to replace $C_j$ by $D_{ij}^3 - D_{ij} B_j$. Comparing with Equation (\ref{equ:eexx}), it follows that
\[
\lambda = - \sum_{j | j \neq i} ( 3D_{ij}^2 - B_j)^2.
\]
\end{proof}

The next result says that the tautological classes with even powers of $e$ can be written as a sum of $n$ terms where the $i$-th term only involves $x_i$.

\begin{proposition}\label{prop:tauteven}
For $i = 1, \dots , n$, we set
\[
p_1(i) = 3x_i^2 + 2B_i, \quad e(i) = 3x_i^2-B_i.
\]
Then for all $a,b \ge 0$ we have
\[
\kappa_{p_1^a e^{2b} } = \int_{E/B} p_1^a e^{2b} = \sum_{i=1}^n \int_{E/B} p_1(i)^a e(i)^{2b}.
\]
\end{proposition}
\begin{proof}
Consider the polynomials
\[
P_1(t_1, \dots , t_n) = 3(t_1^2 + \cdots + t_n^2) + \mu
\]
and
\[ 
P_2(t_1 , \dots , t_n ) = \sum_{i=1}^n (3t_i^2-B_i)^2 + \lambda.
\]
Then
\[
p_1 = P_1(x_1 , \dots , x_n)
\]
by Proposition \ref{prop:p1} and
\[
e^2 = P_2(x_1, \dots , x_n)
\]
by Proposition \ref{prop:e2}. Therefore
\[
\kappa_{p_1^a e^{2b} } = \int_{E/B} p_1^a e^{2b} = \int_{E/B} P_1(x_1, \dots , x_n)^a P_2(x_1 , \dots , x_n)^{b}.
\]
Next we find that
\begin{align*}
(P_1)_i(x_i) &= P_1( D_{i1} , \dots , D_{i (i-1)} , x_i , D_{i (i+1)} , \dots , D_{i n}) \\
&= 3x_i^2 + 3\sum_{j | j \neq i} D_{ij}^2 + \mu \\
&= 3x_i^2 + 2B_i \\
&= p_1(i),
\end{align*}
where the second to last equality was obtained using Proposition \ref{prop:bk}.\\

Also
\begin{align*}
(P_2)_i(x_i) &= P_1( D_{i1} , \dots , D_{i (i-1)} , x_i , D_{i (i+1)} , \dots , D_{i n}) \\
&= (3x_i^2-B_i)^2 + \sum_{ j | j \neq i } (3D_{ij}^2 - B_j)^2 + \lambda \\
&= (3x_i^2-B_i)^2 \\
&= e(i)^2
\end{align*}
where we used Proposition \ref{prop:e2}. Combining these results with the integration formula (Proposition \ref{prop:intf}), we get:
\begin{align*}
\kappa_{p_1^a e^{2b} } &= \int_{E/B} P_1(x_1, \dots, x_n)^a P_2(x_1 , \dots , x_n)^{b} \\
&= \sum_{j=1}^n \int_{E/B} (P_1)_j^a(x_i) (P_2)_j^{b}(x_i) \\
&= \sum_{j=1}^n \int_{E/B} p_1(i)^{a} e(i)^{2b}.
\end{align*}
\end{proof}

Next we would like to consider tautological classes with an odd power of $e$. Since Proposition \ref{prop:e2} lets us write $e^2$ in terms of the $x_i$, we just need to consider integrands of the form $f(x_1 , \dots , x_n)e$ where $f$ is a polynomial. The next proposition gives a formula for the evaluation of such integrals.

\begin{proposition}[Second integration formula]\label{prop:intf2}
Let $f( t_1 , \dots , t_n)$ be a polynomial in $t_1, \dots , t_n$ with coefficients in $H^{ev}(B ; \mathbb{Q})$. Then for each $i \in \{1 , \dots , n\}$, we have
\begin{equation}\label{equ:intformula2}
\int_{E/B} f(x_1 , \dots , x_n)e = \sum_{j=1}^n \int_{E/B} f_j(x_j) e(j) - 2 \sum_{j | j \neq i} f_j(D_{ij}).
\end{equation}
\end{proposition}
\begin{proof}

Fix $i \in \{1, \dots , n\}$. Then from 
\[
x_i^2 = \nu + \sum_{j | j \neq i } D_{ij}x_j + G_i,
\]
we get
\[
\int_{E/B} f(x_1 , \dots , x_n) \nu = \int_{E/B} f(x_1 , \dots , x_n ) ( x_i^2 - \sum_{j | j \neq i } D_{ij} x_j - G_i ).
\]
The integrand on the right hand side is a polynomial in $x_1, \dots , x_n$ with coefficients in $H^{ev}(B ; \mathbb{Q})$, so from Proposition \ref{prop:intf}, we have
\begin{equation}\label{equ:fint1}
\begin{aligned}
\int_{E/B} f(x_1 , \dots , x_n) \nu &= \int_{E/B} f_i(x_i)( x_i^2 - \sum_{j | j \neq i} D_{ij}^2 - G_i) \\
& \; \; \; \; + \sum_{j | j \neq i} f_j(x_j) ( D_{ji}^2 - D_{ij}x_j - \sum_{k | k \neq i,j} D_{ik}D_{jk} - G_i).
\end{aligned}
\end{equation}
Using Proposition \ref{prop:intf} again, we also have
\begin{equation}\label{equ:fint2}
\begin{aligned}
\int_{E/B} f(x_1 , \dots , x_n) ( x_1^2 + \cdots + x_n^2) &= \int_{E/B} f_i(x_i)( x_i^2 + \sum_{j | j \neq i} D_{ij}^2 ) + \sum_{j | j \neq i} \int_{E/B} f_j(x_j)( x_j^2 + \sum_{k | k \neq j } D_{jk}^2)
\end{aligned}
\end{equation}

Now from Proposition \ref{prop:euler}, we have $e = 2\nu + x_1^2 + \cdots + x_n^2$. Together with Equations (\ref{equ:fint1}) and (\ref{equ:fint2}), this implies
\begin{equation}\label{equ:fe}
\begin{aligned}
\int_{E/B} f(x_1 , \dots , x_n) e &= \int_{E/B} f(x_1 , \dots , x_n)( 2\nu + x_1^2 + \cdots + x_n^2) \\
&= \int_{E/B} f_i(x_i) ( 3x_i^2 - 2G_i - \sum_{j | j \neq i} D_{ij}^2) \\ 
& \; \; \; \; + \sum_{j | j \neq i} \int_{E/B} f_j(x_j) ( x_j^2 + 3D_{ji}^2 - 2G_i - 2 \sum_{k | k \neq i,j} D_{ik}D_{jk} + \sum_{k|k \neq i,j} D_{jk}^2 - 2D_{ij} x_j)
\end{aligned}
\end{equation}
Note that
\begin{equation}\label{equ:fiterm}
3x_i^2 - 2G_i - \sum_{j | j \neq i} D_{ij}^2 = 3x_i^2 - B_i = e(i).
\end{equation}
Also
\begin{equation}\label{equ:xjsquared}
\begin{aligned}
& x_j^2 + 3D_{ji}^2 - 2G_i - 2 \sum_{k | k \neq i,j} D_{ik}D_{jk} + \sum_{k|k \neq i,j} D_{jk}^2 - 2D_{ij} x_j \\
&\; \; \; \; = x_j^2 - 2D_{ij}x_j + 3D_{ji}^2 - B_i + D_{ij}^2 + \sum_{k | k \neq i,j} ( D_{ik}^2 - 2D_{ik}D_{jk} + D_{jk}^2).
\end{aligned}
\end{equation}
Recall Equation (\ref{equ:assoc2}):
\[
G_i + G_j + \sum_{k| k \neq i,j} D_{ik}D_{jk} = D_{ij}^2 + D_{ji}^2.
\]
Combining this with $B_i = 2G_i + \sum_{j | j \neq i } D_{ij}^2$, we have
\[
B_i + B_j = 3D_{ij}^2 + 3D_{ji}^2 + \sum_{k | k \neq i,j} ( D_{ik}^2 - 2D_{ik}D_{jk} + D_{jk}^2).
\]
Substituting this into (\ref{equ:xjsquared}), we get that
\begin{equation}\label{equ:fjterm}
\begin{aligned}
& x_j^2 + 3D_{ji}^2 - 2G_i - 2 \sum_{k | k \neq i,j} D_{ik}D_{jk} + \sum_{k|k \neq i,j} D_{jk}^2 - 2D_{ij} x_j \\
& = x_j^2 -2D_{ij}x_j -2D_{ij}^2+B_j \\
& = e(j) -2( x_j^2 + D_{ij} x_j + D_{ij}^2-B_j ).
\end{aligned}
\end{equation}
Substituting Equations (\ref{equ:fiterm}) and (\ref{equ:fjterm}) into Equation (\ref{equ:fe}), we get
\begin{equation}\label{equ:fe2}
\begin{aligned}
\int_{E/B} f(x_1, \dots , x_n) e &= \int_{E/B} f_i(x_i)e(i) + \sum_{j | j \neq i} \int_{E/B} f_j(x_j)( e(j) -2( x_j^2 + D_{ij} x_j + D_{ij}^2-B_j ) ) \\
&= \sum_{j = 1}^n \int_{E/B} f_j(x_j) e(j) - 2 \sum_{j | j \neq i} \int_{E/B} f_j(x_j) ( x_j^2 + D_{ij} x_j + D_{ij}^2-B_j).
\end{aligned}
\end{equation}
Now we claim that for any integer $m \ge 0$, we have
\[
\int_{E/B} x_j^m ( x_j^2 + D_{ij} x_j + D_{ij}^2 - B_j ) = D_{ij}^m.
\]
We prove this by induction on $m$. For $m=0,1$ this is obvious. For $m=2$, we have
\[
\int_{E/B} x_j^2( x_j^2 + D_{ij} x_j + D_{ij}^2 - B_j ) = \left( \int_{E/B} x_j^4 \right) + D_{ij}^2-B_j = B_j + D_{ij}^2 - B_j = D_{ij}^2.
\]
Now suppose $m \ge 3$ and that the result holds for all $m' \le m$. Then since $x_j^3 = B_j x_j + C_j$, and $D_{ij}^3 = B_j D_{ij} + C_j$, we find
\begin{align*}
& \int_{E/B} x_j^m ( x_j^2 + D_{ij}x_j + D_{ij}-B_j) \\
&= \int_{E/B} x_j^{m-3} (B_j x_j + C_j)( x_j^2 + D_{ij}x_j + D_{ij}-B_j) \\
&= B_j \int_{E/B} x_j^{m-2}( x_j^2 + D_{ij}x_j + D_{ij}-B_j) + C_j \int_{E/B} x_j^{m-3}( x_j^2 + D_{ij}x_j + D_{ij}-B_j) \\
&= B_j D_{ij}^{m-2} + C_j D_{ij}^{m-3} \\
&= (B_j D_{ij} + C_j) D_{ij}^{m-3} \\
&= D_{ij}^m,
\end{align*}
which completes the induction. As a consequence, it follows that
\[
\int_{E/B}  f_j(x_j) ( x_j^2 + D_{ij} x_j + D_{ij}^2-B_j) = f_j(D_{ij}).
\]
Applying this to Equation (\ref{equ:fe2}), we get
\[
\int_{E/B} f(x_1 , \dots , x_n) e = \sum_{j=1}^n \int_{E/B} f_j(x_j) e(j) - 2 \sum_{j | j \neq i} f_j( D_{ij} ).
\]
\end{proof}

The next result is the counterpart of Proposition \ref{prop:tauteven} for odd powers of $e$.

\begin{proposition}\label{prop:tautodd}
For all $a,b$, we have
\[
\kappa_{p_1^a e^{2b+1} } = \int_{E/B} p_1^a e^{2b+1} = \sum_{i=1}^n \int_{E/B} p_1(i)^a e(i)^{2b} - 2 \sum_{j | j \neq i} (3D_{ij}^2+2B_j)^a (3D_{ij}^2-B_j)^{2b}.
\]
\end{proposition}
\begin{proof}
As in the proof of Proposition \ref{prop:tauteven}, we write
\[
p_1 = P_1(x_1 , \dots , x_n)
\]
and
\[
e^2 = P_2(x_1, \dots , x_n)
\]
where
\[
P_1(t_1, \dots , t_n) = 3(t_1^2 + \cdots + t_n^2) + \mu
\]
and
\[ 
P_2(t_1 , \dots , t_n ) = \sum_{i=1}^n (3t_i^2-B_i)^2 + \lambda.
\]
Then we apply Proposition \ref{prop:intf2} to $f(x_1 , \dots , x_n) = P_1(x_1, \dots , x_n)^a P_2(x_1 , \dots , x_n)^b$ to obtain
\begin{align*}
\int_{E/B} p_1^a e^{2b+1} &= \int_{E/B} f(x_1 , \dots , x_n) e \\
&= \sum_{j=1}^n \int_{E/B} p_1(j)^a e(j)^{2b+1} - 2 \sum_{j | j \neq i} p_j(D_{ij})^a e_j(D_{ij})^{2b} \\
&= \sum_{j=1}^n \int_{E/B} p_1(j)^a e(j)^{2b+1} - 2 \sum_{j | j \neq i} (3D_{ij}^2+2B_j)^a (3D_{ij}^2-B_j)^{2b}.
\end{align*}

\end{proof}

In the next result, we show that $\mu$ can be written in terms of tautological classes.

\begin{proposition}\label{prop:eep1e}
We have
\[
\kappa_{e^2} = 3\sum_{j=1}^n B_j \text{ and } \kappa_{p_1 e} = 8\kappa_{e^2}/3 + 2\mu.
\]
Therefore
\[
\mu = \dfrac{1}{2} \kappa_{p_1 e} - \dfrac{4}{3} \kappa_{e^2}.
\]
\end{proposition}
\begin{proof}
From Proposition \ref{prop:tauteven}, we have
\begin{align*}
\kappa_{e^2} &= \sum_{j=1}^n \int_{E/B} (3x_j^2 - B_j)^2 \\
&= \sum_{j=1}^n \int_{E/B} ( 9x_j^4 - 6 B_j x_j^2 + B_j^2 ) \\
&= \sum_{j=1}^n (9B_j - 6B_j) = 3 \sum_{j=1}^n B_j.
\end{align*}

From Proposition \ref{prop:tautodd}, we have
\begin{align*}
\kappa_{p_1 e } &= \sum_{j=1}^n \int_{E/B} (3x_j^2+2B_j)(3x_j^2-B_j) - 2 \sum_{j | j \neq i} (3D_{ij}^2+2B_j) \\
&= \sum_{j=1}^n \int_{E/B}( 9x_j^4 + 3B_j x_j^2 - 2B_j^2) - 6\sum_{j | j \neq i} D_{ij}^2  - 4 \sum_{j | j \neq i} B_j \\
&= \sum_{j=1}^n 12 B_j - 4 \sum_{j=1}^n B_j + 4B_i -4B_i + 2\mu \\
&= 8 \sum_{j=1}^n B_j + 2\mu,
\end{align*}
where the second to last equality follows from Proposition \ref{prop:bk}. Therefore
\[
\kappa_{p_1 e} = 8 \sum_{j=1}^n B_j + 2\mu = 8\kappa_{e^2}/3 + 2\mu.
\]
\end{proof}

Let $U_X \to BDiff(X)$ denote the universal family $U_X = EDiff(X) \times_{Diff(X)} X$ over $BDiff(X)$ and let $\overline{U}_X = p^*(U_X)$ be the pullback of the universal family to $\overline{BDiff}(X)$. Then as in Section \ref{sec:families}, the cohomology ring $H^*(\overline{U}_X ; \mathbb{Q})$ is generated over $H^*( \overline{BDiff}(X) ; \mathbb{Q})$ by classes $x_1 , \dots , x_n \in H^2( \overline{U}_X ; \mathbb{Q})$. By the universal coefficient theorem, rational cohomology classes of $\overline{BDiff}(X)$ are detected by their evaluation on integral homology classes. By a result of Thom \cite{tho}, for any integral homology class $x \in H_k( \overline{BDiff}(X) ; \mathbb{Z} )$, there is a non-zero integer $N$ such that $Nx$ is the push-forward of the fundamental class of a compact, oriented smooth manifold $M$ of dimension $k$ under a continuous map $f : M \to \overline{BDiff}(X)$. Hence rational cohomology classes of $\overline{BDiff}(X)$ are detected by continuous maps from compact smooth manifolds into $\overline{BDiff}(X)$. From this it follows that all of the results in Sections \ref{sec:cohomring} and \ref{sec:tautclasses} for smooth, compact families carry over to $\overline{U}_X \to \overline{BDiff}(X)$. In particular, there are classes
\begin{align*}
D_{ij} &\in H^2( \overline{BDiff}(X) ; \mathbb{Q}), \\
G_i &\in H^4( \overline{BDiff}(X) ; \mathbb{Q}), \\
J_i &\in H^6( \overline{BDiff}(X) ; \mathbb{Q}), \\
\omega &\in H^8( \overline{BDiff}(X) ; \mathbb{Q}),
\end{align*}
such that
\begin{equation*}
\begin{aligned}
x_i x_j &= D_{ij}x_i + D_{ji}x_j -D_{ij}D_{ji}, \\
x_i^2 & = \nu + \sum_{j | j\neq i} D_{ij} x_j + G_i, \\
x_i \nu &= G_i x_i - \sum_{j | j \neq i} D_{ij} D_{ji} x_j + J_i, \\
\nu^2 &= \sum_j J_j x_j + \omega.
\end{aligned}
\end{equation*}
We also define $\mu \in H^4( \overline{BDiff}(X) ; \mathbb{Q})$ to be given by $\mu = \frac{1}{2} \kappa_{p_1 e} - \frac{4}{3} \kappa_{e^2}$.

We also have
\[
x_i^3 = B_i x_i + C_i,
\]
where
\begin{align*}
B_i &= 2G_i + \sum_{j | j \neq i } D_{ij}^2 \in H^2(\overline{BDiff}(X) ; \mathbb{Q}), \\
C_i &= J_i - \sum_{j | j \neq i} D_{ij}^2 D_{ji} \in H^3(\overline{BDiff}(X) ; \mathbb{Q}).
\end{align*}

The classes $x_1, \dots , x_n$ (and therefore also the classes $D_{ij}, G_i, J_i, \omega, B_i, C_i$) depend on a choice of framing $\xi_1, \dots , \xi_n$ for the family $\overline{U}_X$. Recall that the group
\[
W_n = S_n \ltimes \mathbb{Z}_2^n
\]
acts on the set of framings by permutations and sign changes. The group $W_n$ acts on $\overline{BDiff}(X)$ and $\overline{U}_X$ on the right inducing left actions on $H^*( \overline{BDiff}(X) ; \mathbb{Q})$ and $H^*(\overline{U}_X ; \mathbb{Q})$. This action corresponds to a change of framing, in particular it follows that
\[
\sigma(x_i) = x_{\sigma(i)} \text{ for } \sigma \in S_n
\]
and
\[
\theta_i (x_j) = \begin{cases} x_j & j \neq i, \\ -x_j & j = i. \end{cases}
\]
It also follows that $W_n$ acts on all of the associated classes $D_{ij}, G_i, J_i, \omega, B_i, C_i$. Noting that
\[
D_{ij} = \int_{E_X/B} x_i^2 x_j,
\]
one finds that the action of $W_n$ on the $D_{ij}$ is given by
\[
\sigma( D_{ij} ) = D_{ \sigma(i) \sigma(j) }, \quad \sigma \in S_n
\]
and
\[
\theta_k ( D_{ij} ) = \begin{cases} D_{ij} & k \neq j, \\ -D_{ij} & k=j. \end{cases}
\]

\begin{definition}
Let $X$ be a smooth, compact, simply-connected, positive definite $4$-manifold with $b_2(X)=n \ge 1$. We denote by $D^*(X)$ the subring of $H^*( \overline{BDiff}(X) ; \mathbb{Q})$ generated by the $D_{ij}$. Note that $W_n$ acts on $D^*(X)$ by ring automorphisms. We let $I^*(X) \subseteq D^*(X)$ denote the $W_n$-invariant subring of $D^*(X)$.
\end{definition}

\begin{remark}\label{rem:invariantbdiff}
Recall that the $W_n$-invariant subring of $H^*( \overline{BDiff}(X) ; \mathbb{Q})$ is $H^*( BDiff(X) ; \mathbb{Q})$. Therefore, $I^*(X)$ may be identified with a subring of $H^*(BDiff(X) ; \mathbb{Q})$
\[
I^*(X) \subseteq H^*( BDiff(X) ; \mathbb{Q}).
\]
\end{remark}

\begin{lemma}
We have that $I^k(X)$ is non-zero only if $k$ is a multiple of $4$. Moreover, $I^4(X)$ is spanned by $I_1, I_2$, where
\[
I_1 = \sum_{i,j | i \neq j} D_{ij}^2, \quad I_2 =  \sum_{i,j,k | i,j,k \text{ distinct}} D_{ik} D_{jk}.
\]
\end{lemma}
\begin{proof}
Since the $D_{ij}$ have degree $2$, $I^*(X)$ is concentrated in even degrees. Suppose $k = 2m$. Any element in $I^{2m}(X)$ is a linear combination of monomials
\[
D_{i_1 j_1} D_{i_2 j_2} \cdots D_{i_m j_m}.
\]
The subgroup $\mathbb{Z}_2^n \subset W_n$ sends each such monomial to plus or minus itself. Therefore any element of $D^{2m}(X)$ must be a linear combination consisting only of monomials that are $\mathbb{Z}_2^n$-invariant. However it is clear that
\[
\theta_k (D_{i_1 j_1} D_{i_2 j_2} \cdots D_{i_m j_m}) = (-1)^{\epsilon_k} D_{i_1 j_1} D_{i_2 j_2} \cdots D_{i_m j_m}
\]
where $\epsilon_k$ is the number of $a \in \{ 1, \dots , m\}$ for which $j_a = k$. Clearly
\[
\sum_{k=1}^n \epsilon_k = m.
\]
This means that the product $\theta_1 \theta_2 \cdots \theta_m$ acts on $D^{2m}(X)$ by $(-1)^m$. Hence $I^{2m}(X)=0$ for $m$ odd and so $I^*(X)$ is concentrated in degrees divisible by $4$.

Any element of $D^4(X)$ is a quadratic polynomial in the $D_{ij}$. Any invariant element of $D^4(X)$ must be a linear combination of monomials $D_{i_1 j_1} D_{i_2 j_2}$ that are $\mathbb{Z}_2^n$-invariant. Such a monomial $D_{i_1 j_1} D_{i_2 j_2}$ is $\mathbb{Z}_2^n$-invariant if and only if $j_1 = j_2$. Thus any element of $I^4(X)$ is a linear combination of monomials of the form
\[
D_{ij}^2, \; (i \neq j) \text{ or } D_{ik}D_{jk}, \;  (i,j,k \text{ distinct}).
\]
The symmetric group $S_n$ acts on such monomials with precisely two orbits. It follows that $I^4(X)$ is spanned by
\[
I_1 = \sum_{i,j | i \neq j} D_{ij}^2, \text{ and } I_2 =  \sum_{i,j,k | i,j,k \text{ distinct}} D_{ik} D_{jk}.
\]
\end{proof}

\begin{lemma}\label{lem:mu}
For $n \ge 2$ we have that
\[
\mu = - \left( \dfrac{n-5}{n(n-1)} \right)I_1 - \left( \dfrac{2}{n(n-1)} \right) I_2.
\]
\end{lemma}
\begin{proof}
From Proposition \ref{prop:bk}, we have
\[
-2B_i + 3 \sum_{j | j \neq i } D_{ij}^2 + \mu = 0
\]
Summing over all $i$, and using Proposition \ref{prop:eep1e}, we have
\begin{equation}\label{equ:nmu}
-\dfrac{2}{3} \kappa_{e^2} + 3I_1 + n \mu = 0.
\end{equation}
Next, recall that $B_i = 2G_i + \sum_{j | j \neq i} D_{ij}^2$. Summing over $i$, we get
\begin{equation}\label{equ:sumgi}
\sum_{i=1}^n G_i = \dfrac{1}{6} \kappa_{e^2} - \dfrac{1}{2} I_1.
\end{equation}
Next, consider equation (\ref{equ:assoc2}):
\[
G_i + G_j + \sum_{k| k \neq i,j} D_{ik}D_{jk} = D_{ij}^2 + D_{ji}^2
\]
for distinct $i,j$. Summing over all $i \neq j$, we get
\begin{equation}\label{equ:sumgi2}
2(n-1) \sum_{i=1}^n G_i + I_2 = 2I_1.
\end{equation}
Combining Equations (\ref{equ:sumgi}) and (\ref{equ:sumgi2}), we get
\[
\dfrac{1}{3}\kappa_{e^2} = \left( \dfrac{n+1}{n-1} \right) I_1 - \left( \dfrac{1}{n-1} \right) I_2.
\]
Note that we can divide by $(n-1)$ because of the assumption that $n \ge 2$. Substituting this into Equation (\ref{equ:nmu}), we get
\[
n \mu = - \left( \dfrac{n-5}{n-1} \right)I_1 - \left( \dfrac{2}{n-1} \right)I_2.
\]

\end{proof}

\begin{proposition}\label{prop:bici}
For $n \ge 2$, we have
\[
B_i = \dfrac{3}{2} \sum_{j | j \neq i} D_{ij}^2 - \dfrac{(n-5)}{2n(n-1)} I_1 - \dfrac{1}{(n-1)} I_2,
\]
\[
C_i = \dfrac{1}{(n-1)} \sum_{j | j \neq i} \left( D_{ji}^3 - B_i D_{ji}\right).
\]
Moreover, $B_i, C_i, G_i \in D^*(X)$ and $\omega \in I^8(X)$.
\end{proposition}
\begin{proof}
Proposition \ref{prop:bk} and Lemma \ref{lem:mu} together imply that
\[
B_i = \dfrac{3}{2} \sum_{j | j \neq i} D_{ij}^2 - \dfrac{(n-5)}{2n(n-1)} I_1 - \dfrac{1}{(n-1)} I_2 \in D^4(X).
\]
Proposition \ref{prop:dcubic} gives $C_i = D_{ji}^3 - B_i D_{ji}$ for all $j \neq i$. Averaging over $j$ for $j \neq i$ gives
\[
C_i = \dfrac{1}{(n-1)} \sum_{j | j \neq i} \left( D_{ji}^3 - B_i D_{ji}\right) \in D^6(X).
\]
(It is not necessary to average over $j$ in order to show $C_i \in D^6(X)$, but this makes the expression for $C_i$ symmetric). Now from Equations (\ref{equ:bg}) and (\ref{equ:cj}), it follows that $G_i, J_i \in D^*(X)$. Lastly if we take Equation (\ref{equ:assoc4}) and average over $i$, we obtain $\omega \in I^8(X)$.
\end{proof}

\begin{proof}[Proof of Theorem \ref{thm:tautring}] We have already constructed the classes
\[
D_{ij} \in H^2( \overline{BDiff}(X) ; \mathbb{Q}),  1 \le i,j  \le n, i \neq j
\]
and the group $W_n$ acts on the $D_{ij}$ as specified in part (i) of the theorem. Part (ii) of the theorem was explained in Remark \ref{rem:invariantbdiff}. For part (iii), first note that the tautological classes are $W_n$-invariant because they lie in $H^*( BDiff(X) ; \mathbb{Q})$, which is the $W_n$-invariant part of $H^*( \overline{BDiff}(X) ; \mathbb{Q})$. Thus it suffices to show that each tautological class $\kappa_{p_1^a e^b}$ belongs to $D^*(X)$. From Propositions \ref{prop:tauteven}, \ref{prop:tautodd}, it follows that each tautological class $\kappa_{p_1^a e^b}$ can be written as a polynomial in $D_{ij}, B_i, C_i$. But from Proposition \ref{prop:bici} we have $B_i , C_i \in D^*(X)$ and of course $D_{ij} \in D^*(X)$. Hence $\kappa_{p_1^a e^b} \in D^*(X)$.
\end{proof}

For each pair of non-negative integers $a,b$ we define a two variable polynomial $\phi_{a,b}(x,y) \in \mathbb{Z}[x,y]$ as follows. Let 
\[
p(z) = z^3-xz-y, \quad p'(z) = 3z^2-x.
\]
Then we define
\[
\phi_{a,b}(x,y) = \dfrac{1}{2\pi i}  \ointctrclockwise \dfrac{ (p'(z)+3x)^a  (p'(z))^b }{p(z)} dz
\]
where the contour encloses all zeros of $p(z)$. Since $(p'(z)+3x)^{a+1} ( p'(z) )^b = (p'(z)+3x)^a p'(z)^{b+1} + 3x (p'(z))^a (p'(z))^b$, we have that
\begin{equation}\label{equ:rec1}
\phi_{a+1,b}(x,y) = \phi_{a,b+1}(x,y) + 3 x \phi_{a,b}(x,y).
\end{equation}

By a direct computation, one finds that
\[
(p'(z))^3 = 3x( p'(z) )^2 + (27y^2-4x^3) + 27 (p(z))^2 + 54 y p(z).
\]
Multiplying both sides by $(p'(z)+3x)^a ( p'(z))^b )$ and taking contour integrals, we see that
\begin{align*}
\phi_{a,b+3}(x,y) &= 3x \phi_{a,b+2}(x,y) + (27y^2-4x^3) \phi_{a,b}(x,y) \\
& \quad \quad + \dfrac{1}{2\pi i}  \ointctrclockwise (p'(z)+3x)^a  (p'(z))^b ( 27p(z) + 54y) dz.
\end{align*}
But the integrand is holomorphic, so the integral is zero, giving the recursive formula
\begin{equation}\label{equ:rec2}
\phi_{a,b+3}(x,y) = 3x \phi_{a,b+2}(x,y) + (27y^2-4x^3) \phi_{a,b}(x,y).
\end{equation}
The recursive relations (\ref{equ:rec1}), (\ref{equ:rec2}) can be used to compute $\phi_{a,b}$ recursively from $\phi_{0,0}, \phi_{0,1}$ and $\phi_{0,2}$, which we now compute. Since $\phi_{a,b}(x,y)$ is a polynomial in $x,y$, it suffices to compute the value of $\phi_{a,b}(x,y)$ as a function of $(x,y)$ on an open subset of $\mathbb{C}^2$. Assume that the discriminant $4x^3 - 27y^2 \neq 0$ so that $p(z)$ has distinct roots $\lambda_1, \lambda_2, \lambda_3$. Then $p(z) = (z-\lambda_1)(z-\lambda_2)(z-\lambda_3)$, where $\lambda_1 + \lambda_2 + \lambda_3 = 0$, $\lambda_1 \lambda_2 + \lambda_1 \lambda_3 + \lambda_2 \lambda_3 = -x$, $\lambda_1 \lambda_2 \lambda_3 = y$. Then by the residue theorem,
\begin{align*}
\phi_{0,0}(x,y) &= \frac{1}{p'(\lambda_1)} + \frac{1}{p'(\lambda_2)} + \frac{1}{p'(\lambda_3)} \\
& = \frac{1}{(\lambda_1-\lambda_2)(\lambda_1-\lambda_3)} + \frac{1}{(\lambda_2-\lambda_1)(\lambda_2-\lambda_3)} + \frac{1}{(\lambda_3-\lambda_1)(\lambda_3-\lambda_2)} \\
& = \frac{1}{(\lambda_1-\lambda_2)(\lambda_2-\lambda_3)(\lambda_3-\lambda_1)}\left( (\lambda_3-\lambda_2) + (\lambda_1 - \lambda_3) + (\lambda_2 - \lambda_1) \right) \\
& = 0.
\end{align*}
By the argument principle
\[
\phi_{0,1}(x,y) = 3
\]
and by the residue theorem again, 
\begin{align*}
\phi_{0,2}(x,y) &= p'(\lambda_1) + p'(\lambda_2) + p'(\lambda_3) \\
& = (\lambda_1-\lambda_2)(\lambda_1-\lambda_3) + (\lambda_2-\lambda_1)(\lambda_2-\lambda_3) + (\lambda_3-\lambda_1)(\lambda_3-\lambda_2) \\
&= \lambda_1^2 + \lambda_2^2 + \lambda_3^2 -\lambda_1 \lambda_2 - \lambda_1 \lambda_3 - \lambda_2 \lambda_3 \\
& = (\lambda_1 + \lambda_2 + \lambda_3)^2 - 3(\lambda_1 \lambda_2 + \lambda_1 \lambda_3 + \lambda_2 \lambda_3) \\
& 3x.
\end{align*}

\begin{proof}[Proof of Theorem \ref{thm:tautclasses}] From Propositions \ref{prop:tauteven}, \ref{prop:tautodd}, we have
\[
\kappa_{p_1^a e^{2b} } = \int_{E/B} p_1^a e^{2b} = \sum_{i=1}^n \int_{E/B} p_1(i)^a e(i)^{2b}
\]
and
\[
\kappa_{p_1^a e^{2b+1} } = \int_{E/B} p_1^a e^{2b+1} = \sum_{i=1}^n \int_{E/B} p_1(i)^a e(i)^{2b} - 2 \sum_{j | j \neq i} (3D_{ij}^2+2B_j)^a (3D_{ij}^2-B_j)^{2b}.
\]
So it remains to show that
\begin{equation}\label{equ:kappai}
\int_{E/B} p_1(i)^a e(i)^{2b} = \phi_{a,b}(B_i , C_i).
\end{equation}
To prove Equation (\ref{equ:kappai}) for all $a,b \ge 0$, it suffices to show that both sides of the expressions satisfy the same recursion relations and same initial conditions. For convenience let us set
\[
\kappa_{a,b,i} = \int_{E/B} p_1(i)^a e(i)^{2b} = \int_{E/B} (3x_i^2+2B_i)^a (3x_i^2 - B_i)^b.
\]
Then we need to show that $\kappa_{a,b,i} = \phi_{a,b}(B_i , C_i)$ for all $a,b \ge 0$ and all $i$. Clearly
\[
\kappa_{0,0,i}= 0 = \phi_{0,0}(B_i,C_i), \quad \kappa_{0,1,i} = 3 = \phi_{0,1}(B_i,C_i)
\]
and
\[
\kappa_{0,2,i} = \int_{E/B} (3x^2_i-B)^2 = \int_{E/B} (9x^4_i-6Bx^2_i+B^2_i) = 3B_i = \phi_{0,2}(B_i,C_i).
\]
So $\kappa_{a,b,i} = \phi_{a,b}(B_i,C_i)$ for $(a,b) = (0,0), (0,1), (0,2)$. Next, from $(p_1)_i = e_i + 3B_i$, we see that
\[
\kappa_{a+1,b,i} = \kappa_{a,b+1,i} + 3B_i \kappa_{a,b,i}.
\]
Lastly, a short calculation shows that
\[
e(i)^3 = 3B_i e(i)^2 + 27C^2_i-4B^3_i
\]
so that
\[
\kappa_{a,b+3,i} = 3B_i \kappa_{a,b+2,i} + (27 C^2_i - 4B^3_i) \kappa_{a,b,i}.
\]
Hence $\kappa_{a,b,i}$ satisfies the same recursive relations and initial conditions as $\phi_{a,b}(B_i,C_i)$, so $\kappa_{a,b,i} = \phi_{a,b}(B_i,C_i)$ for all $a,b \ge 0$ and all $i$.
\end{proof}

\begin{remark}\label{rem:topstruct}
From Lemma \ref{lem:mu} and Proposition \ref{prop:p1}, it follows that $p_1$ can be completely expressed in terms of the classes $x_1, \dots , x_n$ and $\{ D_{ij} \}$. Therefore, $p_1$ depends only on the underlying topological structure of the family $\pi : E \to B$, because the classes $x_1, \dots , x_n$, $\{ D_{ij} \}$ are uniquely characterised in terms of the push-forward map $\pi_* : H^*(E ; \mathbb{Q}) \to H^{*-4}(B ; \mathbb{Q})$ (the classes $x_1, \dots , x_n$, $\{ D_{ij} \}$ also depend on a choice of framing, but $p_1$ is clearly $W_n$-invariant, so does not depend on this choice). It is also clear that $e$ depends only on the underlying topological structure of the family. Therefore, the tautological classes $\kappa_{p_1^a e^b}(E)$ depend only on the topological structure of the family. As mentioned in the introduction, this also follows from the fact that the tautological classes are also defined for topological bundles \cite[Theorem B]{erw}.
\end{remark}

\section{$\mathbb{CP}^2$}\label{sec:n=1}

In this section we specialise to the case $n=1$. Amongst other results, we completely determine the tautological ring of $\mathbb{CP}^2$.

\begin{theorem}\label{thm:n=1case}
Let $E \to B$ be a smooth family with fibres diffeomorphic to $X$, where $X$ is a smooth, compact, simply-connected, positive definite $4$-manifold with $b_2(X)=1$. Suppose that the family has structure group $Diff_0(X)$ and let $\xi \in H^2(X ; \mathbb{Z})$ be a framing. Let $x \in H^2 ( E ; \mathbb{Q})$ be the unique class such that $x|_{X} = \xi$ and $\int_{E/B} x^3 = 0$. Then there exists classes $B \in H^4( B ; \mathbb{Q})$, $C \in H^6( B ; \mathbb{Q})$ such that:
\begin{itemize}
\item[(i)]{There is an isomorphism of $H^*(B ; \mathbb{Q})$-algebras
\[
H^*(E ; \mathbb{Q}) \cong H^*(B ; \mathbb{Q})[x]/( x^3 - Bx-C).
\]
}
\item[(ii)]{The Euler class and first Pontryagin classes of $T(E/B)$ are given by
\[
e = 3x^2-B, \quad p_1 = 3x^2+2B.
\]
}
\item[(iii)]{For all $a,b \ge 0$, we have
\[
\kappa_{p_1^a e^b}(E) = \phi_{a,b}( B , C ).
\]
}
\end{itemize}
\end{theorem}

\begin{proof}
(i) is immediate from Proposition \ref{prop:xcubed}. From Proposition \ref{prop:bk}, we have $\mu = 2B$. Then Proposition \ref{prop:p1} gives
\[
p_1 = 3x^2 + \mu = 3x^2 + 2B.
\]
From Equation (\ref{equ:bg}) and Proposition \ref{prop:ring}, we have 
\[
x^2 = \nu + G
\]
where $G = B/2$. Then from Proposition \ref{prop:euler}, we have
\[
e = 2\nu + x^2 = (2x^2-B)+x^2 = 3x^2 - B.
\]
This proves (ii). From (ii), it follows that
\[
\kappa_{p_1^a e^b}(E) = \int_{E/B} (3x^2+2B)^a(3x^2-B)^b.
\]
Using the exact same argument as in the Proof of Theorem \ref{thm:tautclasses}, we have that
\[
\kappa_{p_1^a e^b}(E) = \phi_{a,b}(B,C).
\]
\end{proof}

Note in particular that
\[
\kappa_{p_1^2}(E) = 21B, \quad \kappa_{p_1^4}(E) = 81 C^2 + 609 B^3.
\]

\begin{theorem}\label{thm:tautringcp2}
The tautological ring of $\mathbb{CP}^2$ is isomorphic to a polynomial ring generated by $\kappa_{p_1^2}$ and $\kappa_{p_1^4}$:
\[
R^*(\mathbb{CP}^2) \cong \mathbb{Q}[ \kappa_{p_1^2}, \kappa_{p_1^4}].
\]
\end{theorem}
\begin{proof}

As explained in the introduction, if a relation amongst tautological classes holds in $R^*(E)$ for all smooth compact $\mathbb{CP}^2$ families $E \to B$, then it must also hold in $R^*(\mathbb{CP}^2)$. Furthermore, since the map $H^*( BDiff( \mathbb{CP}^2) ; \mathbb{Q} ) \to H^*( BDiff_0(\mathbb{CP}^2) ; \mathbb{Q})$ is injective, we can further restrict to families with structure group $Diff_0(\mathbb{CP}^2)$. From Theorem \ref{thm:n=1case}, we see that every tautological class $\kappa_{p_1^a e^b}(E)$ is a polynomial in $B$, $C$. In fact, by comparing degrees we see that only even powers of $C$ can occur and hence $\kappa_{p_1^a e^b}(E)$ is a polynomial in $B$ and $C^2$. Next, since
\[
\kappa_{p_1^2}(E) = 21B, \quad \kappa_{p_1^4}(E) = 81 C^2 + 609 B^3,
\]
we see that $B$ and $C^2$ can be expressed as
\[
B = \dfrac{1}{21} \kappa_{p_1^2}(E), \quad C^2 = \dfrac{1}{81} \kappa_{p_1^4}(E) - \dfrac{203}{27} ( \kappa_{p_1^2}(E)/21 )^3.
\]
Hence every tautological class can be written as a polynomial in $\kappa_{p_1^2}, \kappa_{p_1^4}$. To complete the proof it remains to check that there are no relations between $\kappa_{p_1^2}$ and $\kappa_{p_1^4}$. To show this, consider families of the form $E = \mathbb{P}(V)$, the bundle of projective spaces underlying a complex rank $3$ vector bundle of the form $V = L \oplus M \oplus (L^*M^*)$ for two line bundles $L,M \to B$. If $c_1(L) = l$, $c_1(M) = m$, then one finds
\[
B = l^2 + m^2 + lm, \quad C = lm(l+m).
\]
So $-B, -C$ are the second and third elementary symmetric polynomials in $l , m , -l-m$. It follows that there can be no relation between $B,C$ that holds for all line bundles $L,M$ on all $B$ and hence there can be no relation between $\kappa_{p_1^2}$ and $\kappa_{p_1^4}$.
\end{proof}

Following \cite{gri}, \cite{ggrw}, we consider variants $R^*( X , *)$ and $R^*( X , D^4)$ of the tautological ring, which are defined as follows. Let $Diff(X , *)$ be the subgroup of $Diff(X)$ fixing a point and $Diff(X , D^4)$ the subgroup which acts as the identity on an open disc $D^4 \subset X$. There are obvious inclusions
\begin{equation}\label{equ:inclusions}
Diff(X , D^4 ) \to Diff(X , * ) \to Diff(X)
\end{equation}
and a homomorphism $s : Diff(X , *) \to GL_+(4,\mathbb{R})$ which sends a diffeomorphism of $X$ to its derivative at the marked point. For each $c \in H^*(BGL_+(4,\mathbb{R}) ; \mathbb{Q}) \cong H^*(BSO(4) ; \mathbb{Q})$ we can take its pullback $s^*(c) \in H^*(BDiff(X,*) ; \mathbb{Q})$. We define $R^*(X , *)$ to be the subring of $H^*(BDiff(X , *) ; \mathbb{Q})$ generated by the $s^*(c)$ together with the pullback to $BDiff(X , *)$ of all tautological classes $\kappa_c$. We similarly define $R^*(X , D^4)$ to be the subring of $H^*( BDiff(X , D^4) ; \mathbb{Q})$ generated by the pullback to $BDiff(X , D^4)$ of all tautological classes $\kappa_c$. The inclusions (\ref{equ:inclusions}) give ring homomorphisms
\[
R^*(X) \buildrel f \over \longrightarrow R^*( X , *) \buildrel g \over \longrightarrow R^*(X , D^4)
\]
whose composition is surjective.

\begin{theorem}
We have ring isomorphisms
\begin{itemize}
\item[(1)]{$R^*(\mathbb{CP}^2 , * ) \cong \mathbb{Q}[ p_1, e]$.}
\item[(2)]{$R^*(\mathbb{CP}^2 , D^4) \cong \mathbb{Q}$.}
\end{itemize}
\end{theorem}
\begin{proof}
First note that we can identify $BDiff(\mathbb{CP}^2 , *) \to BDiff(\mathbb{CP}^2)$ with the universal bundle $U_X \to BDiff(\mathbb{CP}^2)$ because $Diff(\mathbb{CP}^2 , *)$ is a closed subgroup of $Diff(\mathbb{CP}^2)$, and so
\begin{align*}
BDiff(\mathbb{CP}^2 , * ) &= EDiff(\mathbb{CP}^2)/ Diff(\mathbb{CP}^2 , * ) \\
&= EDiff(\mathbb{CP}^2) \times_{Diff(\mathbb{CP}^2)} ( Diff(\mathbb{CP}^2)/Diff(\mathbb{CP}^2,*)) \\
&= EDiff(\mathbb{CP}^2) \times_{Diff(\mathbb{CP}^2)} \mathbb{CP}^2 = U_X.
\end{align*}
So we can think of $R^*(\mathbb{CP}^2 , *)$ as the subring of $H^*( U_X ; \mathbb{Q})$ generated by $p_1, e$ and the pullback of all tautological classes. Note that the pullback 
\[
\pi^* : H^*(BDiff(\mathbb{CP}^2) ; \mathbb{Q}) \to H^*( U_X ; \mathbb{Q})
\]
is injective because $\int_{U_X/BDiff(\mathbb{CP}^2)} \pi^*( w ) e =  \chi( \mathbb{CP}^2) w = 3w$ for any $w \in H^*(BDiff(\mathbb{CP}^2) ; \mathbb{Q})$. So the tautological classes pulled back to $U_X$ generate a ring isomorphic to $\mathbb{Q}[u,v]$, where $u = \kappa_{p_1^2}, v = \kappa_{p_1^4}$, by Theorem \ref{thm:tautringcp2}. From Theorem \ref{thm:n=1case}, we have that $e = p_1 - 3B = p_1 - \kappa_{p_1^2}/7$. Therefore $R^*( \mathbb{CP}^2 , *)$ is generated by $\kappa_{p_1^2}, \kappa_{p_1^4}$ and $p_1$. Next, one can check directly from Theorem \ref{thm:n=1case} that
\begin{equation}\label{equ:p13}
p_1^3 = \dfrac{4}{7} \kappa_{p_1^2} p_1^2 - \dfrac{5}{49} (\kappa_{p_1^2}) p_1 - \dfrac{17}{1029} (\kappa_{p_1^2})^3 + \dfrac{1}{3} \kappa_{p_1^4}.
\end{equation}
We claim that $R^*(\mathbb{CP}^2 , *)$ is a free $R^*(\mathbb{CP}^2)$-module with basis $1,p_1,p_1^2$. The fact that $R^*(\mathbb{CP}^2 , *)$ is generated by $\kappa_{p_1^2}, \kappa_{p_1^4}$ and $p_1$ together with Equation (\ref{equ:p13}) implies that $R^*(\mathbb{CP}^2,*)$ is generated as an $R^*(\mathbb{CP}^2)$-module by $1,p_1,p_2$. We need to check linear independence. Suppose that 
\begin{equation}\label{equ:lindep}
a(u,v) + b(u,v) p_1 + c(u,v) p_1^2 = 0
\end{equation}
for some $a,b,c \in \mathbb{Q}[u,v]$. Note that
\[
\kappa_{p_1^2} = u, \quad \kappa_{p_1^3} = \dfrac{13}{49} u^2, \quad \kappa_{p_1^4} = v.
\]
Integrating over the fibres, we get $3b+uc = 0$, so $b = -uc/3$. Multiplying (\ref{equ:lindep}) by $p_1$ and integrating, we get $3a - u^2 c/3 + 13 u^2 c/49 = 0$, hence $a = 10 u^2 c/441$. Multiplying (\ref{equ:lindep}) by $p_1^2$ and integrating, we get
\[
0 = \dfrac{10}{441} u^3 c - \dfrac{13}{147} u^3 c + vc = \left( v - \dfrac{29}{441}u^3 \right)c.
\]
Hence $c=0$, which also implies $a=b=0$, proving the claim that $1,p_1,p_1^2$ are linearly independent over $R^*(\mathbb{CP}^2)$. Thus
\[
R^*(\mathbb{CP}^2 , *) \cong \mathbb{Q}[ \kappa_{p_1^2} , \kappa_{p_1^4} , p_1]/(p_1^3 - \dfrac{4}{7} \kappa_{p_1^2} p_1^2 + \dfrac{5}{49} (\kappa_{p_1^2}) p_1 + \dfrac{17}{1029} (\kappa_{p_1^2})^3 - \dfrac{1}{3} \kappa_{p_1^4}).
\]
Using the relation (\ref{equ:p13}), we can solve for $\kappa_{p_1^4}$ in terms of $\kappa_{p_1^2}$, hence $R^*(\mathbb{CP}^2 , *) \cong \mathbb{Q}[\kappa_{p_1^2}, p_1]$. Then using $e = p_1 - \kappa_{p_1^2}/7$, we have that $R^*(\mathbb{CP}^2 , *) \cong \mathbb{Q}[\kappa_{p_1^2}, p_1] \cong \mathbb{Q}[p_1 , e]$.

Consider the ring $R^*(\mathbb{CP}^2 , D^4)$. Since the composition $R^*(\mathbb{CP}^2) \to R^*(\mathbb{CP}^2 , * ) \to R^*(\mathbb{CP}^2 , D^4)$ is surjective, to show $R^*(\mathbb{CP}^2 , D^4) = \mathbb{Q}$, it suffices to show that the image of $R^*(\mathbb{CP}^2 , *) \to R^*(\mathbb{CP}^2 , D^4)$ is $\mathbb{Q}$. Recall that $g : R^*( \mathbb{CP}^2 , *) \to R^*(\mathbb{CP}^2 , D^4)$ is the homomorphism induced by $BDiff(\mathbb{CP}^2 , *) \to BDiff(\mathbb{CP}^2 , D^4)$. It follows that $g(p_1) = g(e) = 0$ because the composition $Diff(\mathbb{CP}^2 , D^4 ) \to Diff(\mathbb{CP}^2 , *) \to SO(4)$ is a constant map. But we have just shown that $R^*(\mathbb{CP}^2 , *) \cong \mathbb{Q}[p_1,e]$, hence the image of $R^*(\mathbb{CP}^2 , *) \to R^*(\mathbb{CP}^2 , D^4)$ is $\mathbb{Q}$, as claimed.
\end{proof}

\section{$\mathbb{CP}^2 \# \mathbb{CP}^2$}\label{sec:n=2}

In this section, we specialise to the case that $n = b_2(X) = 2$ and set $D_1 = D_{12}$, $D_2 = D_{21}$. 

\begin{lemma}
Each class $\kappa_{p_1^a e^b}$ is a symmetric polynomial in $D_1^2, D_2^2$.
\end{lemma}
\begin{proof}
From Proposition \ref{prop:xcubed}, we have
\[
x_1^3 = B_1 x_1 + C_1, \quad x_2^3 = B_2 x_2 + C_2,
\]
where
\[
B_1 = 2G_1 + D_1^2, \quad B_2 = 2G_2 + D_2^2, \quad C_1 = J_1 - D_1^2 D_2, \quad C_2 = J_2 - D_2^2 D_1.
\]
Equation (\ref{equ:assoc2}) gives $G_1 + G_2 = D_1^2 + D_2^2$ and hence
\begin{equation}\label{equ:plus}
B_1 + B_2 = 3D_1^2 + 3D_2^2.
\end{equation}
Proposition \ref{prop:bk}, gives
\[
2B_1 - 3D_1^2 = \mu = 2B_2 - 3D_2^2
\]
and hence
\begin{equation}\label{equ:minus}
2(B_1-B_2) = 3D_1^2 - 3D_2^2. 
\end{equation}
Equations (\ref{equ:plus}) and (\ref{equ:minus}) give
\[
B_1 = \dfrac{9}{4}D_1^2 + \dfrac{3}{4} D_2^2, \quad B_2 = \dfrac{9}{4}D_2^2 + \dfrac{3}{4} D_1^2.
\]
From Proposition \ref{prop:dcubic}, we get
\[
C_1 = D_2^3 - D_2 B_1 = D_2^3  - \dfrac{9}{4}D_1^2 D_2 - \dfrac{3}{4}D_2^3 = \dfrac{1}{4} D_2 ( D_2^2 - 9D_1^2).
\]
Similarly
\[
C_2 = \dfrac{1}{4} D_1 ( D_1^2 - 9D_2^2).
\]
Therefore
\[
C_1^2 = \dfrac{1}{16} D_2^2 (D_2^2 - 9D_1^2)^2, \quad C_2^2 = \dfrac{1}{16} D_1^2 ( D_1^2 - 9D_2^2)^2.
\]
Notice that $B_1,B_2,C_1^2,C_2^2$ are all polynomials in $D_1^2, D_2^2$. Then from Propositions \ref{prop:tauteven} and \ref{prop:tautodd}, it follows that each tautological class $\kappa_{p_1^a e^b}$ is a polynomial in $D_1^2, D_2^2$. By averaging if necessary, we have that $\kappa_{p_1^a e^b}$ is given by a symmetric polynomial in $D_1^2, D_2^2$.
\end{proof}

From Proposition \ref{prop:tauteven}, we have
\[
\kappa_{p_1^2} = 21(B_1+B_2) = 63( D_1^2 + D_2^2).
\]
and
\[
\kappa_{p_1^3} = 117( B_1^2 + B_2^2 ) = \dfrac{1053}{8} \left( 5(D_1^2+D_2^2) - 4D_1^2 D_2^2 \right).
\]

It follows that the tautological ring contains all symmetric polynomials in $D_1^2, D_2^2$ and is generated by $\kappa_{p_1^2}, \kappa_{p_1^3}$.

\begin{theorem}
The tautological ring of $\mathbb{CP}^2 \# \mathbb{CP}^2$ is isomorphic to a polynomial ring generated by $\kappa_{p_1^2}$ and $\kappa_{p_1^3}$:
\[
R^*(\mathbb{CP}^2 \# \mathbb{CP}^2) \cong \mathbb{Q}[ \kappa_{p_1^2}, \kappa_{p_1^3}].
\]
\end{theorem}
\begin{proof}

As in the proof of Theorem \ref{thm:tautringcp2}, if a relation amongst tautological classes holds in $R^*(E)$ for all smooth compact $\mathbb{CP}^2$ families $E \to B$ with structure group $Diff_0(\mathbb{CP}^2 \# \mathbb{CP}^2)$, then it holds in the tautological ring $R^*( \mathbb{CP}^2 \# \mathbb{CP}^2)$. We have already seen that the tautological ring is generated by $\kappa_{p^2}$ and $\kappa_{p_1^3}$, or equivalently by $D_1^2+D_2^2$ and $D_1^2 D_2^2$. So it remains to show that there are no relations between $D_1$ and $D_2$.

Consider first a $\mathbb{CP}^2$ family of the form $E_1 = \mathbb{P}(V_1)$ and $V_1$ is a complex rank $3$ vector bundle of the form $V_1 = L \oplus M \oplus \mathbb{C}$ for two line bundles $L,M \to B$. This family has an obvious section $s_1 : B \to E_1$ corresponding to the $\mathbb{C}$ summand of $V_1$. The normal bundle of $s_1$ is $L \oplus M$. Similarly, let $E_2 = \mathbb{P}(V_2)$, where $V_2 = L \oplus M^* \oplus \mathbb{C}$. Then $E_2$ has an obvious section $s_2 : B \to E_2$ corresponding to the $\mathbb{C}$ summand of $V_2$. The normal bundle of $s_2$ is $L \oplus M^*$. Note that the underlying real vector bundles of $L \oplus M$ and $L \oplus M^*$ are isomorphic, but inherit opposite orientations from the complex structures on $L$ and $M$. Therefore, we can remove tubular neighbourhoods of $s_1$ and $s_2$ from $E_1$ and $E_2$ and identify the boundaries of the resulting spaces by an orientation reversing diffeomorphism to obtain the families connected sum $E = E_1 \#_B E_2$. This is a smooth compact family with fibres diffeomorphic to $\mathbb{CP}^2 \# \mathbb{CP}^2$. Let $l = c_1(L)$, $m = c_1(M)$. A straightforward calculation shows that $D_1, D_2$ for the family $E$ are given by
\[
D_1 = (l-m)/3, \quad D_2 = (l+m)/3.
\]
It follows that there can be no relation between $D_1,D_2$ that holds for all line bundles $L,M$ and hence there can be no relation between $\kappa_{p_1^2}$ and $\kappa_{p_1^3}$.
\end{proof}

\section{Linear relations in the tautological ring}\label{sec:linrel}

\begin{proof}[Proof of Theorem \ref{thm:linrel}] 
We first note that each of the polynomials $\phi_{a,b}(x,y)$ involves only even powers of $y$. This is clear from the recursive relations and initial conditions satisfied by the $\phi_{a,b}$. Note also that if we set $deg(x) = 2$, $deg(y) = 3$, then $\phi_{a,b}(x,y)$ is a homogeneous polynomial of degree $2(a+b-1)$. The space of homogeneous polynomials in $x,y^2$ of degree $2(d-1)$ has dimension $1 + \lfloor (d-1)/3 \rfloor$. On the other hand there are $1 + \lfloor d/2 \rfloor$ pairs $(a,b)$ with $a+b=d$ and $b$ even. Hence there are at least
\[
\lfloor d/2 \rfloor - \lfloor (d-1)/3 \rfloor
\]
linear relations amongst the polynomials $\phi_{a,b}(x,y)$ with $a+b=d$ and $b$ even. Any such linear relation may be written in the form
\[
\sum_{j=0}^{\lfloor d/2 \rfloor} c_j \phi_{d-2j,2j}(x,y) = 0
\]
for some $c_0 , c_1 , \dots , c_{\lfloor d/2 \rfloor} \in \mathbb{Q}$. Now from Theorem \ref{thm:tautclasses}, we have
\begin{align}
\sum_{j=0}^{\lfloor d/2 \rfloor} c_j \kappa_{ p_1^{d-2j} e^{2j} } &=  \sum_{j=0}^{\lfloor d/2 \rfloor} c_j \sum_{i=1}^n \phi_{d-2j , 2j}(B_i , C_i) \\
&= \sum_{i=1}^n \sum_{j=0}^{\lfloor d/2 \rfloor} c_j \phi_{d-2j , 2j}(B_i , C_i) \\
&=0.
\end{align}
Hence we have at least $\lfloor d/2 \rfloor - \lfloor (d-1)/3 \rfloor$ linear relations amongst the tautological classes $\kappa_{p_1^a e^b}$ with $b$ even and $a+b=d$.
\end{proof}

\begin{remark}
Since $X$ is definite, $H^2(X ; \mathbb{R}) \cong H^+(X)$ is a trivial bundle and $H^-(X) = 0$. So the families signature theorem implies that each component of $\int_{E/B} \mathcal{L}( T(E/B) )$ of positive degree must vanish, where $\mathcal{L}$ denotes the $L$-polynomial. Each component of $\mathcal{L}( T(E/B))$ is a polynomial in $p_1$ and $p_2 = e^2$, and so each component of $\int_{E/B} \mathcal{L}( T(E/B) )$ is a tautological class. Equating these to zero gives linear relations in the tautological ring. From \cite[Page 3864]{rw}, the first few such relations (up to $d = 9$) are:
\begin{equation*}
\begin{aligned}
0& = \kappa_{p_1^2} - 7 \kappa_{e^2}, \\
0 & = 2\kappa_{p_1^3} - 13 \kappa_{p_1 e^2}, \\
0 &= 3 \kappa_{p_1^4} - 22 \kappa_{p_1^2 e^2} + 19 \kappa_{e^4}, \\
0 &= 10 \kappa_{p_1^5} - 83 \kappa_{p_1^3 e^2} + 127 \kappa_{p_1 e^4}, \\
0 &= 1382 \kappa_{p_1^6} - 12842 \kappa_{p_1^4 e^2} +27635 \kappa_{p_1^2 e^4} - 8718 \kappa_{e^6}, \\
0 &= 420 \kappa_{p_1^7} - 4322 \kappa_{p_1^5 e^2} + 11880 \kappa_{p_1^3 e^4} - 7978 \kappa_{p_1 e^6}, \\
0 &= 10851 \kappa_{p_1^8} - 122508 \kappa_{p_1^6 e^2} +407726 \kappa_{p_1^4 e^4} - 423040 \kappa_{p_1^2 e^6} +68435 \kappa_{e^8}, \\
0 &= 438670 \kappa_{p_1^9}  - 5391213 \kappa_{p_1^7 e^2} + 20996751\kappa_{p_1^5 e^4}  - 29509334 \kappa_{p_1^3 e^6} + 11098737 \kappa_{p_1 e^8}.
\end{aligned}
\end{equation*}
In general, for each $d \ge 2$, we obtain one linear relation amongst the tautological classes $\kappa_{p_1^a e^b}$ with $a+b=d$ and $b$ even. Theorem \ref{thm:linrel} implies that there are further linear relations whenever $ \lfloor d/2 \rfloor - \lfloor (d-1)/3 \rfloor  > 1$. This is the case if $d=6$ or $d \ge 8$. By a direct computation, we find the first few such relations (up to $d=12$) are:
\begin{align*}
0 &= 4\kappa_{p_1^4 e^2} - 41 \kappa_{p_1^2 e^4} + 100 \kappa_{e^6}, \\
0 &= 36 \kappa_{p_1^6 e^2} - 461 \kappa_{p_1^4 e^4} + 1843 \kappa_{p_1^2 e^6} - 2300 \kappa_{e^8}, \\
0 &= 24 \kappa_{p_1^7 e^2} - 322 \kappa_{p_1^5 e^4} + 1379 \kappa_{p_1^3 e^6} - 1900 \kappa_{p_1 e^8}, \\
0 &= 108 \kappa_{p_1^8 e^2}-1579 \kappa_{p_1^6 e^4} + 7902 \kappa_{p_1^4 e^6} - 15531 \kappa_{p_1^2 e^8} + 9100 \kappa_{e^{10}}, \\
0 &= 360 \kappa_{p_1^9 e^2}-5606 \kappa_{p_1^7 e^4}+30923 \kappa_{p_1^5 e^6} -71311 \kappa_{p_1^3 e^8} + 57100 \kappa_{p_1 e^{10}}, \\
0 &= 144 \kappa_{p_1^8 e^4}-2552 \kappa_{p_1^6 e^6}+16629 \kappa_{p_1^4 e^8} - 47400 \kappa_{p_1^2 e^{10}} + 50000 \kappa_{e^{12}}, \\
0 &= 6000 \kappa_{p_1^{10} e^2} -98012 \kappa_{p_1^8 e^4} + 577796 \kappa_{p_1^6 e^6} -1461667 \kappa_{p_1^4 e^8} + 1338700 \kappa_{p_1^2 e^{10}}.
\end{align*}
\end{remark}


\bibliographystyle{amsplain}

\end{document}